\documentclass[reqno]{amsart}
\usepackage{amsmath}%
\usepackage{amsfonts}%
\usepackage{amssymb}%
\usepackage{graphicx}
\newtheorem{theorem}{Theorem}
\theoremstyle{plain}

\newtheorem{corollary}{Corollary}

\newtheorem{definition}{Definition}
\newtheorem{example}{Example}

\newtheorem{lemma}{Lemma}

\newtheorem{proposition}{Proposition}
\newtheorem{remark}{Remark}

\numberwithin{equation}{section}
\begin{document}
\begin{center}
\vspace*{1.3cm}

\textbf{PROPER EFFICIENCY AND CONE EFFICIENCY}

\bigskip

by

\bigskip

PETRA WEIDNER\footnote{HAWK Hochschule f\"ur angewandte Wissenschaft und Kunst Hildesheim/\-Holzminden/ G\"ottingen University of Applied Sciences and Arts, Faculty of Natural Sciences and Technology,
D-37085 G\"ottingen, Germany, {petra.weidner@hawk-hhg.de}.}

\bigskip
\bigskip
Research Report \\ 
May 10, 2017

\end{center}

\bigskip
\bigskip

\noindent{\small {\textbf{Abstract:}}
In this report, two general concepts for proper efficiency in vector optimization are studied. Properly efficient elements can be defined as minimizers of functionals with certain monotonicity properties or as weakly efficient elements with respect to sets that contain the domination set. Interdependencies between both concepts are proved in topological vector spaces by means of Gerstewitz functionals. The investigation includes proper efficiency notions introduced by Henig and by Nehse and Iwanow. In contrary to Henig's notion, proper efficiency by Nehse and Iwanow is defined as efficiency with respect to certain convex sets which are not necessarily cones.
For the finite-dimensional case, we turn to Geoffrion's proper efficiency as a special case of Henig's proper efficiency. It is characterized as efficiency with regard to subclasses of the set of polyhedral cones. Conditions for the existence of Geoffrion's properly efficient points are proved. For closed feasible point sets, Geoffrion's properly efficient point set is empty or coincides with that of Nehse and Iwanow. Properly efficient elements by Nehse and Iwanow are the minimizers of continuous convex functionals with certain monotonicity properties. Henig's proper efficiency can be described by means of minimizers of continuous sublinear functionals with certain monotonicity properties.
}

\bigskip

\noindent{\small {\textbf{Keywords:}} 
Vector optimization; multicriteria optimization; weak efficiency; scalarization}

\bigskip

\noindent{\small {\textbf{Mathematics Subject Classification (2010): }
90C29, 46N10}}

\section{Introduction}

Consider the vector optimization problem given as
\[\tag{VOP}
D\mbox{--}\min_{x\in S} f(x)
\]
under the assumption that $S$ is an arbitrary nonempty set, $f:S\rightarrow Y$, and that $D$ is a nonempty subset of $Y$. Throughout this paper, $Y$ is assumed to be a real topological vector space.

$D$ is the set which defines the solution concept for (VOP). Imagine that for each $y^0\in F:=f(S)$ the set of elements in $F$ that is preferred to $y^0$ is just $F\cap (y^0-(D\setminus\{0\}))$. Then we are interested in the set 
$$\operatorname*{Min}(F,D):=\{y^{0}\in F\mid F\cap (y^0-D)\subseteq\{y^0\}\}$$ 
of efficient points of $F$ with regard to (w.r.t.) $D$. We will call $D$ the domination set of (VOP).

The author (\cite{Wei83}, \cite{Wei85}, \cite{Wei90}, \cite{Wei16b}) studied vector optimization problems under such general assumptions motivated by decision theory, especially by ideas from Yu \cite{yu74}. Related bibliographical notes are given in \cite{Wei16b}. If $D$ is an ordering cone in $Y$, $\operatorname*{Min}(F,D)$ is the set of elements of $F$ that are minimal w.r.t. the cone order $\leq_D$. In the literature, vector optimization problems are usually defined with domination sets that are ordering cones.

In general, it is easier to determine solutions to vector optimization problems w.r.t. open domination sets.
We define $\operatorname*{WMin}(F,D):=\operatorname*{Min}(F,\operatorname*{int}D)$ as the set of 
weakly efficient elements of $F$ w.r.t. $D$, where $\operatorname*{int}D$ denotes the topological interior of $D$.
In many vector optimization problems, weakly efficient points can be characterized as minimal solutions of some scalar functions, whereas efficiency of a point can only be shown if the minimizer of a functional is unique, which is difficult to check. Since the weakly efficient point set can be much more comprehensive than the efficient point set, one looks for possibilities to guarantee efficiency of solutions. This is the reason why different notions of proper efficiency came into existence. Properly efficient point sets are subsets of the efficient point set and usually defined in such a way that they can be easier determined than efficient points in general. Properly efficient points can be used to approximate the efficient point set if the 
properly efficient point set is dense in the efficient point set. This paper focuses on proper efficiency that is weak efficiency w.r.t. other domination sets. The proper efficiency notions of Henig \cite{Hen82}, of Nehse and Iwanow \cite{iwne84} and of Geoffrion \cite{geof68} fit into this framework. This offers the possibility to apply the scalarization results for weakly efficient elements which were given in \cite{GerWei90}, \cite{Wei90} and in \cite{Wei16b} to proper efficiency.

After some preliminaries in Section \ref{s-pre}, we will start our investigation in Section \ref{s-propeff} with proper efficiency for the general vector optimization problem described above. In Section \ref{s-Geoff}, we turn to multicriteria optimization problems and study Geoffrion's proper efficiency. 
It is characterized as efficiency with regard to subclasses of the set of polyhedral cones. Conditions for the existence of Geoffrion's properly efficient points are proved. For closed feasible point sets, Geoffrion's properly efficient point set is empty or coincides with that of Nehse and Iwanow. The results from Lemma \ref{l-DK} to Example \ref{ex-NInotHe-folg}, which refer to the characterization of Geoffrion's proper efficiency by polyhedral cones and by minimizers of functionals, were first given by the author in \cite{Wei90}.

\section{Preliminaries}\label{s-pre}

From now on, $\mathbb{R}$ and $\mathbb{N}$ will denote the set of real numbers and of nonnegative integers, respectively.
We define 
$\mathbb{N}_{>}$ as the set of positive integers,
$\mathbb{R}_{+}:=\{x\in\mathbb{R} \mid x\geq 0\}$, $\mathbb{R}_{>}:=\{x\in\mathbb{R} \mid x > 0\}$,
$\mathbb{R}_{+}^{\ell }:=\{(x_1,\ldots ,x_{\ell })^T\in\mathbb{R}^{\ell }\mid x_i\geq 0 \,\forall i\in\{1,\ldots, \ell \}\}$ for each $\ell\in\mathbb{N}_>$.  $\overline{\mathbb{R}}:=\mathbb{R}\cup \{-\infty ,+\infty\}$ denotes the extended real-valued set. 
Given some set $B\subseteq\mathbb{R}$, $d\in Y$, and $D\subseteq Y$, we write $B d:=\{b \cdot d \mid b\in B\}$ and $B D:=\{b \cdot d \mid b\in B, d\in D\}$.
A set $C$ in $Y$ is said to be a cone iff $\lambda c\in C \mbox{ for all } \lambda\in\mathbb{R}_{+}, c\in C.$ The cone $C$ is called nontrivial iff $C\not=\emptyset$, $C\not=\{0\}$ and $C\not= Y$ hold. It is said to be pointed iff $C\cap (-C)=\{0\} .$ 
Let $A$ be a subset of $Y$.
$0^+A:=\{u\in Y   \mid  A+\mathbb{R}_{+}u\subseteq A\}$ denotes the recession cone of $A$.
$\operatorname*{{core}}A$ stands for the algebraic interior of $A$, $\operatorname*{conv}A$ for the convex hull of $A$. Furthermore, $\operatorname*{cl}A$ and $\operatorname*{bd}A$ denote the closure and the boundary, respectively, of $A$.
Consider a functional $\varphi :Y\to \overline{\mathbb{R}}$ and its effective domain 
$\operatorname*{dom}\varphi :=\{y\in Y \mid \varphi(y)\in\mathbb{R}\cup \{-\infty\}\}$.
$\varphi$ is said to be finite-valued on $A$ iff it attains only real values on $A$. It is called finite-valued iff it is finite-valued on $Y$. According to the rules of convex analysis, $\operatorname*{inf}\emptyset =+\infty .$ Moreover, the following functional turns out to be essential for characterizing solutions in vector optimization. 

\begin{definition}
Assume $A\subseteq Y$ and $k\in Y\setminus\{0\}$.\\ 
The Gerstewitz functional $\varphi _{A,k}:Y\rightarrow \overline{{\mathbb{R}}}$ is defined by
\begin{equation*}
\varphi_{A,k} (y):= \inf \{t\in
{\mathbb{R}} \mid y\in A+tk\}. 
\end{equation*}
\end{definition}

For properties and bibliographical notes related to this functional, see \cite{Wei17a} and \cite{Wei17b}.

Functionals which are applied for scalarization in vector optimization have to fulfill certain monotonicity conditions.
\begin{definition}
Suppose $B\subseteq Y$  and $\varphi: Y \to \overline{\mathbb{R}}$, $M\subseteq \operatorname*{dom}\varphi$. \\
$\varphi$ is said to be
\begin{itemize}
\item[\rm (a)]
$B$-monotone on $M$
iff $y^1,y^2 \in M$ and $y^{2}-y^{1}\in B$ imply $\varphi
(y^{1})\leq \varphi (y^{2})$,
\item[\rm (b)] strictly $B$-monotone on $M$ 
iff $y^1,y^2 \in M$ and $y^{2}-y^{1}\in B\setminus
\{0\}$ imply $\varphi (y^{1})<\varphi (y^{2})$.
\end{itemize}
$\varphi$ is said to be $B$-monotone or strictly $B$-monotone iff it is $B$-monotone or strictly $B$-monotone, respectively, on $\operatorname*{dom}\varphi$.
\end{definition}

\section{Subsets of the Efficient Point Set and Proper Efficiency}\label{s-propeff}

From now on, we assume $F\subseteq Y$ and that $D$ is a proper subset of $Y$ with $D\setminus\{0\}\not=\emptyset$.

The properties of efficient point sets yield two general concepts of proper efficiency.
Since each element of $F$ in which some strictly $D$--monotone functional attains its minimum on $F$ is an efficient element of $F$ w.r.t. $D$, such minimal solutions are appropriate for defining proper efficiency. 

\begin{definition}
Suppose that $\Phi $ is a nonempty subset of the set of functions $\varphi :F\to\mathbb{R}$ which are strictly $D$-monotone on $F$.
$y^0\in F$ is said to be a $\Phi $--properly efficient element of $F$ w.r.t. $D$ iff there exists a function $\varphi \in \Phi $ that attains its minimum on $F$ in $y^0$.
\end{definition}

A special case contained in this definition is proper efficiency according to Bitran and Magnanti \cite{bima79}. They defined, for non-trivial convex cones $D\subset Y$, that 
$y^0\in F$ is a properly efficient point of $F$ w.r.t. $D$ if there exists some linear continuous strictly $D$-monotone function $\varphi :Y\to\mathbb{R}$ that attains its minimum on $F$ in $y^0$. If $Y$ is the Euclidean space and $D$ the non-negative orthant in this space, the proper efficiency by Bitran and Magnanti coincides with the proper efficiency by Sch\"onfeld \cite{scho70}.

Another general concept for proper efficiency is given in the following way (\cite{Wei90},\cite{GerWei90}):

\begin{definition}
Suppose $Z$ to be a nonempty subset of the family of sets $H\subseteq Y$ with the property $H\supseteq D\setminus \{ 0 \} $.
$y^0\in F$ is said to be a $Z$--properly efficient element of $F$ w.r.t. $D$ iff there exists some set $H\in Z$ with $y^0\in \operatorname*{Min}(F,H)$.
\end{definition}

Because of $H\supseteq D\setminus \{ 0 \}$, each of these properly efficient points is an efficient element of $F$ w.r.t. $D$.
If $H$ is open for each $H\in Z$, then $\operatorname*{Min}(F,H)=\operatorname*{WMin}(F,H)$. We have already pointed out the advantage of dealing with sets of weakly efficient elements.

We get for
\[ Z=Z_{\operatorname*{He}}:=\{ H\subseteq Y \mid H \mbox{ convex cone},\,
D\setminus \{ 0 \}\subseteq\operatorname*{int}H\}, \]
the notion of proper efficiency by Henig \cite{Hen82};
for
\[ Z=Z_{\operatorname*{NI}}:=\{ \operatorname*{int}H\mid H\subseteq Y \mbox{ closed convex set},\, 0\in bd\, H,\,
H+(D\setminus \{ 0 \} ) \subseteq \operatorname*{int}H \}, \]
we get the notion of proper efficiency by Nehse und Iwanow (\cite{iwne84},\cite{geiw85}).

We have to mention that all authors defined their proper efficiency notions under more restrictive assumptions to the space $Y$ and to the domination sets $D$. The following notions of proper efficiency in (a) and (c) were originally defined in $Y=\mathbb{R}^{\ell }$ w.r.t. a non-trivial convex cone $D$ that had to be closed in Benson's definition, which was given in \cite{bens79}.

\begin{definition}\label{d-Henig}
\begin{itemize}
\item[]
\item[\rm (a)] $y^0\in F$ is said to be a properly efficient element of $F$ w.r.t. $D$ according to Henig iff $y^0\in \operatorname*{Min}(F,H)$ for some convex cone $H\subseteq Y$ with $D\setminus\{0\}\subseteq\operatorname*{int}H$. We will denote the set of these points by $\operatorname*{He-PMin}(F,D)$.
\item[\rm (b)] $y^0\in F$ is said to be a properly efficient element of $F$ w.r.t. $D$ according to Nehse and Iwanow iff $y^0\in \operatorname*{WMin}(F,H)$ for some closed convex set $H\subseteq Y$ with $0\in\operatorname{bd}H$ and $H+(D\setminus\{0\})\subseteq\operatorname*{int}H$. We will denote the set of these points by $\operatorname*{NI-PMin}(F,D)$.
\item[\rm (c)] If $0\in D$, then $y^0\in F$ is said to be a properly efficient element of $F$ w.r.t. $D$ according to Benson iff $\operatorname*{cl}\operatorname*{cone}(F+D-y^0)\cap (-D)=\{0\}$. We will denote the set of these points by $\operatorname*{Be-PMin}(F,D)$.
\end{itemize}
\end{definition}

\begin{remark}
The condition $0\in\operatorname{bd}H$ in part (b) of the definition was added by Z{\u{a}}linescu \cite{zali87} who proved that, without this condition, the properly efficient point set of each set $F$ would be empty or $F$.
\end{remark}

Note that the assumption $0\in D$ in (c) guarantees efficiency of the points in $\operatorname*{Be-PMin}(F,D)$ \cite{Wei83}.

We get immediately from the properties of convex sets:
\begin{lemma}\label{l-NI-open}
$y^0\in \operatorname*{NI-PMin}(F,D)$ holds
if and only if there exists some open convex set $H\subseteq Y$
with $0\in \operatorname*{bd}H$, $\operatorname*{cl}H+(D\setminus \{ 0 \})\subseteq H$
and $y^0\in \operatorname*{Min}(F,H)$.
\end{lemma} 

Z{\u{a}}linescu \cite[Theorem 6]{zali87} proved:
\begin{lemma}
Suppose $Y=\mathbb{R}^{\ell }$ and $D$ to be a closed convex cone with $\operatorname*{int}D\neq \emptyset $.
Then $y^0\in \operatorname*{NI-PMin}(F,D)$ is equivalent to the existence of some closed convex set $H\subseteq \mathbb{R}^{\ell }$ with $D\setminus \{ 0 \}\subseteq \operatorname*{int}H$
and $y^0\in \operatorname*{WMin}(F,H)$.
\end{lemma}

Henig's proper efficiency can also be formulated using weakly efficient elements w.r.t. $H$.

\begin{proposition}\label{l-He-NI}
The following statements are equivalent to each other:
\begin{itemize}
\item[\rm (a)] $y^0\in \operatorname*{He-PMin}(F,D)$.
\item[\rm (b)] $y^0\in\operatorname*{WMin}(F,H)$ for some convex cone $H\subseteq Y$ with $D\setminus\{0\}\subseteq \operatorname*{int}H$.
\item[\rm (c)] $y^0\in\operatorname*{WMin}(F,H)$ for some closed convex cone $H\subseteq Y$ with $D\setminus\{0\}\subseteq \operatorname*{int}H$.
\end{itemize}
Hence, $\operatorname*{He-PMin}(F,D)\subseteq \operatorname*{NI-PMin}(F,D)$.
\end{proposition}

\begin{proof}
(a) implies (b) because of $\operatorname*{Min}(F,H)\subseteq \operatorname*{WMin}(F,H)$.\\
If $H\in Z_{\operatorname*{He}}$, then $H^0:=\operatorname*{int}H\cup\{0\}\in Z_{\operatorname*{He}}$ and $\operatorname*{WMin}(F,H)=\operatorname*{Min}(F,H^0)$. Hence, (a) holds.\\
The equivalence between (b) and (c) results from ${\operatorname*{int}\operatorname*{cl}H=\operatorname*{int}H}$ and \linebreak
$\operatorname*{WMin}(F,H)=\operatorname*{WMin}(F,\operatorname*{cl}H)$ for convex sets $H$ with nonempty interior.
\end{proof}

In general, $\operatorname*{He-PMin}(F,D)\not= \operatorname*{NI-PMin}(F,D)$. This is pointed out by the following example from \cite[Remark 2]{iwne84}.

\begin{example}\label{ex-NInotHe}
Consider, in $Y=\mathbb{R}^2$, the set $F:=\{ (y_1,y_2)^T \in \mathbb{R}^2\mid y_1<0,\, y_2=\frac{1}{y_1}\}+\mathbb{R}^2_+$. Then $(-1,-1)^T\in\operatorname*{Min}(F,\mathbb{R}^2_+)$ and $(-1,-1)^T\in\operatorname*{WMin}(F,H)$ for $H:=- (Y\setminus \operatorname*{int}F)-(1,1)^T$. $H\in Z_{\operatorname*{NI}}$ for $D=\mathbb{R}^2_+$, hence $y^0\in \operatorname*{NI-PMin}(F,\mathbb{R}^2_+)$, but $\operatorname*{He-PMin}(F,\mathbb{R}^2_+)=\emptyset$. 
\end{example} 

Obviously, we have:

\begin{lemma}\label{l-Benson-AF}
Suppose $0\in D$ and $D+D\subseteq D$. Then
$\operatorname*{Be-PMin}(A,D)=\operatorname*{Be-PMin}(F,D)$ for each set $A\subseteq Y$ with $F\subseteq A\subseteq F+D$.
\end{lemma}

\begin{proposition}\label{p-Henig-dense}
Suppose $Y=\mathbb{R}^{\ell }$ and $D$ to be a non-trivial closed pointed convex cone. Then 
\begin{itemize}
\item[\rm (a)] $\operatorname*{He-PMin}(F,D)=\operatorname*{Be-PMin}(F,D)$.
\item[\rm (b)] If $\;\operatorname*{He-PMin}(F,D)\not=\emptyset$ and $A$ is closed for some set $A\subseteq Y$ with $F\subseteq A\subseteq F+D$, then $\operatorname*{He-PMin}(F,D)$ is dense in $\operatorname*{Min}(F,D)$.
\end{itemize}
\end{proposition}

Proposition \ref{p-Henig-dense} was proved by Henig \cite[Theorem 2.1, Theorem 5.1, Theorem 5.2]{Hen82} for sets $A=F+B$ with $0\in B\subseteq D$. Its extension results from Lemma \ref{l-Benson-AF}.

In Section 3.2.6 of \cite{GopRiaTamZal:03}, conditions are given under which $\operatorname*{He-PMin}(F,D)$ is dense in the efficient point set if $Y$ is a normed space and $D$ is a closed convex cone. Note that density of $\operatorname*{He-PMin}(F,D)$ in the efficient point set implies density of $\operatorname*{NI-PMin}(F,D)$ in the efficient point set by Proposition \ref{l-He-NI}.

\begin{theorem}\label{hab-t512}
Assume that $H$ is a closed proper subset of $Y$ with $0\in\operatorname*{bd}H$ and with $H+\mathbb{R}_{>}k\subseteq\operatorname*{int}H$for some $k\in Y\setminus\{0\}$.\\
Suppose $y^0\in \operatorname*{WMin}(F,H)$.
\begin{itemize}
\item[\rm (a)] There exists some functional $\varphi: Y\to\overline{\mathbb{R}}$ with 
$$\varphi (y^0)=\operatorname*{min}_{y\in F}\varphi (y)=0$$
which is continuous on $\operatorname*{dom}\varphi$.\\
$\varphi_{y^0-H,k}$ is such a functional. $\operatorname*{int}H=\{y\in Y\mid \varphi _{y^0-H,k}(y^0-y)<0\}$.
\item[\rm (b)] There exists some functional $\varphi: Y\to\overline{\mathbb{R}}$ with 
$$\varphi (y^0-y^0)=\operatorname*{min}_{y\in F}\varphi (y-y^0)=0$$  
which is continuous on $\operatorname*{dom}\varphi$.\\
$\varphi_{-H,k}$ is such a functional. $\operatorname*{int}H=\{y\in Y\mid \varphi _{-H,k}(-y)<0\}$.
\end{itemize}
Moreover, these functionals have the following properties:
\begin{itemize}
\item[\rm (i)]  If $Y=\operatorname*{bd}H+\mathbb{R}k$, then $\varphi_{y^0-H,k}$ and $\varphi_{-H,k}$ are finite-valued. 
\item[\rm (ii)] If $H+D\subseteq H$, then $\varphi_{y^0-H,k}$ and $\varphi_{-H,k}$ are $D$-monotone. 
\item[\rm (iii)] Assume $H+(D\setminus\{0\})\subseteq\operatorname*{core}H$. If $\varphi_{y^0-H,k}$ or $\varphi_{-H,k}$ is finite-valued on $F$, then it is strictly $D$-monotone on $F$. 
\item[\rm (iv)] $\varphi_{y^0-H,k}$ and $\varphi_{-H,k}$ are convex if $H$ is convex.
\item[\rm (v)] $\varphi_{-H,k}$ is sublinear if $H$ is a convex cone.
\end{itemize}
\end{theorem}

\begin{proof}
$y^0\in \operatorname*{WMin}(F,H)$ implies by \cite[Theorem 7]{Wei16b}:
\begin{eqnarray*}
\varphi _{y^0-H,k}(y^0) & = & \operatorname*{min}_{y\in F}\varphi _{y^0-H,k}(y)=0 \mbox{ and}\\
\varphi _{-H,k} (y^0-y^0) & = & \operatorname*{min}_{y\in F}\varphi _{-H,k}(y-y^0)=0.
\end{eqnarray*} 
The continuity of these functionals and their other properties mentioned in the theorem result from the Theorems 3.1, 2.16 and 2.9 in \cite{Wei17a}.
\end{proof}

\begin{lemma}\label{suff-hab-t512}
The assumptions for $H$ and $k$ in Theorem \ref{hab-t512} are fulfilled and the constructed functionals are finite-valued if one of the following conditions holds:
\begin{itemize}
\item[\rm (a)] $H$ is a closed convex proper subset of $Y$ with $0\in\operatorname*{bd}H$, $\operatorname*{int}H\not=\emptyset$, $k\in 0^+H$ and $Y=H+\mathbb{R}k$.
\item[\rm (b)] $H\subset Y$ is a non-trivial closed convex cone with $k\in\operatorname*{int}H$.
\item[\rm (c)] $H$ is a closed proper subset of $Y$ with $0\in \operatorname*{bd}H$ and $D$ is a non-trivial cone with $k\in\operatorname*{int}D$
and $H+\operatorname*{int}D\subseteq H$.
\end{itemize}
\end{lemma}

\begin{proof}
The statement related to (a) follows from  \cite[Proposition 4.5]{Wei17a} together with  \cite[Theorem 3.1]{Wei17a}.
(b) implies (a), where $Y=H+\mathbb{R}k$ was proved in \cite{Zal:86a}.
Condition (c) implies $k\in\operatorname*{int}0^+H$ by \cite[Prop. 3.13]{Wei17a} and $Y=D+\mathbb{R}k$,
thus by \cite[Corollary 3.12]{Wei17a} and by \cite[Theorem 3.1]{Wei17a} the assertion.
\end{proof}

Under an assumption that is equivalent to the condition given in Lemma \ref{suff-hab-t512}(a), the statement of Theorem \ref{hab-t512} was given in \cite[Theorem 3.4]{GerWei90}. It was also proved in \cite[Theorem 3.5]{GerWei90} under an assumption which is sufficient for the condition in Lemma \ref{suff-hab-t512}(c).

Furthermore, we have \cite[Corollary 3.2]{GerWei90}: 

\begin{corollary}\label{cor-jota1}
Assume $Y=D+\mathbb{R} k$ for some $k\in Y\setminus\{0\}$ with $\mathbb{R}_>k\subseteq D$.\\
Then $y^0\in \operatorname*{NI-PMin}(F,D)$ if and only if $y^0\in F$ is a point in which some continuous
convex strictly $D$--monotone functional $\varphi :Y\to \mathbb{R}$ attains its minimum on $F$.
\end{corollary}

For $Y=\mathbb{R}^{\ell }$ and $D=\mathbb{R}_+^{\ell }$, Iwanow und Nehse
\cite[Theorem 2]{iwne84} proved the statement of Corollary \ref{cor-jota1}.

The proof of \cite[Corollary 3.2]{GerWei90} and the properties of functionals $\varphi_{A,k}$ yield the following statement.

\begin{corollary}\label{c-NI-H}
Assume $Y=D+\mathbb{R} k$ for some $k\in Y\setminus\{0\}$ with $\mathbb{R}_>k\subseteq D$.\\
Then $y^0\in \operatorname*{NI-PMin}(F,D)$ if and only if there exists some closed convex set $H\subseteq Y$ with $0\in \operatorname*{bd}H$ and $H+(D\setminus \{ 0 \})\subseteq \operatorname*{int}H$ such that (a) or (b) holds. Note that (a) and (b) are equivalent to each other.
\begin{itemize}
\item[\rm (a)] $\varphi_{y^0-H,k}(y^0)=\operatorname*{min}_{y\in F}\varphi_{y^0-H,k}(y)=0$.
\item[\rm (b)] $\varphi_{-H,k}(y^0-y^0)=\operatorname*{min}_{y\in F}\varphi_{-H,k}(y-y^0)=0$.
\end{itemize}
Both functionals are finite-valued, continuous, convex and strictly $D$--monotone.
\end{corollary}

The assumptions of Corollary \ref{cor-jota1} and of Corollary \ref{c-NI-H} are  fulfilled if
$D\subset Y$ is a non-trivial cone with $k\in \operatorname*{core}D$.
Because of \cite[Proposition 4]{Zal:86a}, the assumptions are also satisfied if
$0\in D$, $D$ is convex, $\operatorname*{core}D\neq
\emptyset $ and $k\in \operatorname*{core}0^+(\operatorname*{cl}D)$. 

We have proved an analogous statement for nonconvex functionals and sets \cite[Corollary 3.3]{GerWei90}:

\begin{corollary}\label{cor-jota2}
Assume that $D\subset Y$ is a non-trivial cone with $\operatorname*{int}D\neq \emptyset $.\\
Then $y^0\in F$ is a point in which some continuous strictly $D$--monotone functional $\varphi : Y\to\mathbb{R}$ attains its minimum on $F$
if and only if 
$y^0\in \operatorname*{Min}(F,H)$ for some open set $H\subseteq Y$
with $0\in \operatorname*{bd}H$ and $\operatorname*{cl}H+(D\setminus \{ 0 \})\subseteq H$.
\end{corollary}

\begin{corollary}\label{Henig_minima}
\begin{itemize}
\item[]
\item[\rm (a)] $y^0\in \operatorname*{He-PMin}(F,D)$ if and only if $y^0\in F$ is a point for which there exists some continuous sublinear strictly $D$--monotone functional $\varphi :Y\to \mathbb{R}$
with $\varphi (y^0-y^0)=\operatorname*{min}_{y\in F}\varphi (y-y^0)=0$.
\item[\rm (b)] $y^0\in \operatorname*{He-PMin}(F,D)$ if and only if there exists some non-trivial closed convex cone $H\subseteq Y$ with $D\setminus \{ 0 \}\subseteq \operatorname*{int}H$ such that $\varphi_{-H,k}(y^0-y^0)=\operatorname*{min}_{y\in F}\varphi_{-H,k}(y-y^0)=0$, where $k\in \operatorname*{int}H$ can be chosen arbitrarily.
$\varphi_{-H,k}$ is finite-valued, continuous, sublinear and strictly $D$--monotone.
\end{itemize}
\end{corollary}

\begin{proof}
\begin{itemize}
\item[]
\item[\rm (a)] The forward direction follows from Theorem \ref{hab-t512} with Lemma \ref{suff-hab-t512}.\\
Suppose now that $y^0\in F$ is a point for which there exists some continuous
sublinear strictly $D$--monotone functional $\varphi :Y\to \mathbb{R}$
with $\varphi (y^0-y^0)=\operatorname*{min}_{y\in F}\varphi (y-y^0)=0$. Apply \cite[Theorem 3.3]{GerWei90} to $\tilde{\varphi } :Y\to \mathbb{R}$ given by $\tilde{\varphi }(y):=\varphi (y-y^0)\;\forall y\in Y$. Then just $\varphi (y)=\tilde{\varphi }(y+y^0)\;\forall y\in Y$, and we get $y^0\in\operatorname*{Min}(F,H)$ for some open set $H$ for which 
$\operatorname*{cl}H+(D\setminus \{ 0 \})\subseteq H$ holds and $H\cup\{ 0\}$ is a convex cone. Since
$\operatorname*{int}\operatorname*{cl}H=\operatorname*{int}H$, we get $y^0\in\operatorname*{WMin}(F,\bar{H})$ for $\bar{H}:=\operatorname*{cl}H$, where $\bar{H}$ is a non-trivial closed convex cone with $D\setminus \{ 0 \}\subseteq \operatorname*{int}\bar{H}$.\\
\item[\rm (b)] follows from the proof of (a) by the construction of the functional in Theorem \ref{hab-t512}.
\end{itemize}
\end{proof}

Z{\u{a}}linescu \cite{zali87} proved the statement of the above corollary under the assumption that $D$ is a convex cone and $Y$ a separated topological vector space.

\section{Proper Efficiency according to Geoffrion}\label{s-Geoff}

We now turn to proper efficiency defined by Geoffrion
\cite{geof68}, which turns out to be of basic importance for procedures in multicriteria optimization.

Throughout this section, we will assume $Y=\mathbb{R}^\ell$ with $\ell\geq 2$.
We are interested in the sets $\operatorname*{Min}(F):=\operatorname*{Min}(F,\mathbb{R}^\ell _+)$ and
$\operatorname*{WMin}(F):=\operatorname*{WMin}(F,\mathbb{R}^\ell _+)$ of Pareto--optima and weak Pareto--optima, respectively, of $F$.

\begin{definition}
$y^0\in F$ is called a properly efficient element of
$F$ according to Geoffrion iff there exists some
$K\in\mathbb{R}_{>}$ such that, for each $y\in F$ with
$y_i<y_i^0$ for
some $i\in \{ 1,\ldots ,\ell\}$, there exists some $j\in \{ 1,\ldots ,\ell\}
\setminus \{ i \}$ such that $y_i^0-y_i\leq K(y_j-y_j^0)$.
$\operatorname*{GMin}(F)$ will denote the set of these elements in
$F$.
\end{definition}

Benson \cite[Theorem 3.2]{bens79} proved that his proper efficiency notion generalizes Geoffrion's proper efficiency. This implies together with Proposition \ref{p-Henig-dense}:

\begin{proposition}
$\operatorname*{Be-PMin}(F,\mathbb{R}^\ell _+)=\operatorname*{He-PMin}(F,\mathbb{R}^\ell _+)=\operatorname*{GMin}(F)$.
\end{proposition}

Obviously, $\operatorname*{GMin}(F)\subseteq \operatorname*{Min}(F)$. In Example \ref{ex-NInotHe}, $\operatorname*{GMin}(F)=\emptyset$, though Pareto-optima of $F$ exist.

Lemma \ref{l-Benson-AF} yields:
\begin{proposition}\label{hab-s528FA}
$\operatorname*{GMin}(A)=\operatorname*{GMin}(F)$ for each set $A\subseteq \mathbb{R}^\ell$ with $F\subseteq A\subseteq F+\mathbb{R}^\ell _+$.
\end{proposition}

We get from Proposition \ref{p-Henig-dense}:
\begin{theorem}\label{hab-s528}
If $A$ is closed for some set $A\subseteq \mathbb{R}^\ell$ with $F\subseteq A\subseteq F+\mathbb{R}^\ell _+$
and $\operatorname*{GMin}(F)\neq\emptyset$, then
$\operatorname*{GMin}(F)$ is dense in $\operatorname*{Min}(F)$, i.e.,
$\operatorname*{GMin}(F)\subseteq \operatorname*{Min}(F)\subseteq \operatorname*{cl}\operatorname*{GMin}(F)$.
\end{theorem}

A constructive proof of this statement is given in \cite[Proposition 2]{weid93}. 

\begin{remark}
Podinovskij and Nogin \cite[3.1]{pono82} proved $\operatorname*{Min}(F)\subseteq \operatorname*{cl}\operatorname*{GMin}(F)$ for
closed convex sets $F$ on the one hand, and for closed
sets $F$ for which there exist some $u\in\mathbb{R}^\ell$ and some $w\in \operatorname*{int}\mathbb{R}_+^\ell$ with $F\subseteq
u+\{y\in\mathbb{R}^\ell \mid w^Ty\geq 0\}$ on the other hand.
\end{remark}

The characterization of Henig's proper efficiency in Section \ref{s-propeff} results in the following proposition.

\begin{proposition}\label{p-Geoff-cones}
The following statements are equivalent to each other:
\begin{itemize}
\item[\rm (a)] $y\in\operatorname*{GMin}(F)$.
\item[\rm (b)] $y\in \operatorname*{Min}(F,H)$ for some convex cone $H$ with $\mathbb{R}^\ell _+\setminus\{0\}\subseteq\operatorname*{int}H$.
\item[\rm (c)] $y\in\operatorname*{WMin}(F,H)$ for some convex cone $H$ with $\mathbb{R}^\ell _+\setminus\{0\}\subseteq \operatorname*{int}H$.
\item[\rm (d)] $y\in\operatorname*{WMin}(F,H)$ for some closed convex cone $H$ with $\mathbb{R}^\ell _+\setminus\{0\}\subseteq \operatorname*{int}H$.
\end{itemize}
\end{proposition}

In order to describe properly efficient elements according to Geoffrion as efficient elements w.r.t. special cones, we introduce, for each $i\in \{ 1,\ldots ,\ell \}$ and $K\in\mathbb{R}_{>}$, the set
\[ D^{i,K}:=\{ y\in \mathbb{R}^\ell  \mid  y_i>0,\, y_i+Ky_j>0\quad \forall j\in
\{ 1,\ldots ,\ell\}
\setminus \{ i \} \}\]
and, moreover, the set
\[ D^K:=\bigcup\limits_{i\in \{ 1,\ldots ,\ell \} } D^{i,K}
\quad\qquad
\forall K\in\mathbb{R}_{>}. \]

Let us first investigate properties of these sets.

\begin{lemma}\label{l-DK}
Assume $K\in\mathbb{R}_{>}$, $i\in \{ 1,\ldots ,\ell \}$.
\begin{itemize}
\item[\rm (a)] $D^{i,K}\cup \{ 0 \} $ is a convex cone,
$D^{i,K}\cap -D^{i,K}=\emptyset $,\\
$\operatorname*{cl}D^{i,K}+\mathbb{R}_+^\ell\subseteq \operatorname*{cl}D^{i,K}$,
$\operatorname*{cl}D^{i,K}+(\mathbb{R}_+^\ell\setminus \{ 0 \} )\not\subset int\,
D^{i,K}=D^{i,K}$.
\item[\rm (b)] $D^K\cup \{ 0 \}$ is a cone. It is convex if and only if $\ell =2$ and $K\geq 1$.
\item[\rm (c)] $\operatorname*{cl}D^K+(\mathbb{R}_+^\ell\setminus \{ 0 \} )\subseteq int\,
D^K$, hence $\mathbb{R}_+^\ell \setminus \{ 0 \} \subseteq int\, D^K$.
\end{itemize}
\end{lemma}

\begin{proof}
\begin{itemize}
\item[]
\item[\rm (a)] and the cone property in (b) are obvious.\\
For $Y=\mathbb{R}^2$, it is easy to see that $D^K\cup\{ 0
\}$ is convex if and only if $K\geq 1$.\\
Assume now $\ell >2$, $K\in\mathbb{R}_{>}$.
We define $y,z\in D^K$ by\\
$y_i:=\left\{
\begin{array}{c@{\mbox{ if }}l}
3K & i=1,\\
-2 & i\in \{2,\ldots ,\ell\},
\end{array}
\right.\quad$
and
$\quad z_i:=\left\{
\begin{array}{c@{\mbox{ if }}l}
-4 & i=1,\\
7K & i=2,\\
-6 & i\in \{3,\ldots ,\ell\}.
\end{array}
\right.$\\
$\begin{array}[t]{lcl}
y+z\notin D^K, \mbox{ since }& & (y_1+z_1)+K(y_3+z_3)=-5K-4<0,\\
 & & (y_2+z_2)+K(y_3+z_3)=-K-2<0 \mbox{ and} \\
 & & y_j+z_j<0\quad \forall j\in
\{3,\ldots ,\ell\} .
\end{array}$\\
Thus, $D^K$ is not convex.
\item[\rm (c)] Consider $y\in \operatorname*{cl}D^K$, $z\in \mathbb{R}_+^\ell\setminus \{ 0 \}$.\\
$\begin{array}[t]{lcl}
 &\Rightarrow& \exists i\in \{1,\ldots ,\ell\} \,:\;  y_i\geq 0 \mbox{
and }y_i+Ky_j\geq 0
\quad \forall j\in \{1,\ldots ,\ell\}\setminus \{ i \} ,\\
 & & \mbox{and }\exists n\in \{1,\ldots ,\ell\} \,:\;  z_n>0 \mbox{ and } z_j\ge
0\quad \forall j\in \{1,\ldots ,\ell\}.
\end{array}$\\
First case: $y_i=0$.\\
$\Rightarrow y_j\ge 0 \quad \forall j\in \{1,\ldots ,\ell\}$.
$\Rightarrow y+z\in D^{n,K}\subseteq D^K$.\\
Second case: $y_i>0$.\\
$\Rightarrow y_i+z_i>0$ and
$(y_i+z_i)+K(y_j+z_j)=(y_i+Ky_j)+z_i+Kz_j>0$
for all
$j\in \{1,\ldots ,\ell\}\setminus \{ i \} $.
\end{itemize}
\end{proof}

\begin{proposition}\label{hab-s521}
\begin{equation*}
\operatorname*{GMin}(F)  =  \bigcup\limits_{K>0}\operatorname*{Min}(F,D^K)
				  =  \bigcup\limits_{K>0}\: \bigcap\limits_{i\in \{1,\ldots
,\ell\} }\operatorname*{Min}(F,D^{i,K}),
\end{equation*}
where we have for $0<K<\bar{K}$: \;
$D^{\bar{K}}\subseteq D^K$ and $D^{i,\bar{K}}\subseteq D^{i,K}
\quad \forall i\in \{1,\ldots ,\ell\}$, and thus\\
${\operatorname*{Min}(F,D^K)\subseteq \operatorname*{Min}(F,D^{\bar{K}})}$ and
${\operatorname*{Min}(F,D^{i,K})\subseteq \operatorname*{Min}(F,D^{i,\bar{K}})} \; {\forall
i\in \{1,\ldots ,\ell\}}$.
\end{proposition}

\begin{proof}
\begin{tabular}[t]{lcl}
$y^0\in \operatorname*{GMin}(F)$ & $\iff$ & $y^0\in F \mbox{ and } \exists K>0\quad \forall y\in F \mbox{ with } y_i^0>y_i \mbox{ for}$\\
 & & $\mbox{some } i\in \{1,\ldots ,\ell\}:$\\
 & & $\exists j\in \{1,\ldots ,\ell\}\setminus \{ i \}: \quad y_i^0-y_i\le K(y_j-y_j^0),$\\
 & $\iff$ & $y^0\in F \mbox{ and } \exists K>0\quad \forall y\in F:\quad y^0-y\notin D^K,$\\
 & $\iff$ & $\exists K>0:\quad y^0\in \operatorname*{Min}(F,D^K),$\\
 & $\iff$ & $y^0\in\bigcup\limits_{K>0}\operatorname*{Min}(F,D^K).$
\end{tabular}\\
\begin{tabular}{l}
$\operatorname*{Min}(F,D^K) =  \operatorname*{Min}(F,\bigcup\limits_{i\in
\{1,\ldots ,\ell\} }D^{i,K})=\bigcap\limits_{i\in \{1,\ldots ,\ell\} }
\operatorname*{Min}(F,D^{i,K})\quad\forall K>0.$\\
\end{tabular}\\
\begin{tabular}{l}
$0<K<\bar{K}. \Rightarrow D^{i,\bar{K}}\subseteq D^{i,K}\quad\forall i\in \{1,\ldots ,\ell\}, \mbox{ and } D^{\bar{K}}\subset D^K.$\\
$\Rightarrow E\!f\!f(F,D^{i,K})\subseteq \operatorname*{Min}(F,D^{i,\bar{K}})\quad\forall i\in \{1,\ldots ,\ell\}, \mbox{ and }$ \\
$\operatorname*{Min}(F,D^K)\subseteq \operatorname*{Min}(F,D^{\bar{K}}).$
\end{tabular}\\
\end{proof}

Moreover, we get:

\begin{proposition}\label{hab-s523}
For each convex cone $H$ with $\mathbb{R}_+^\ell\setminus \{ 0 \}\subseteq \operatorname*{int}H$, there exists some $K\in\mathbb{R}_>$
with $D^K\subseteq \operatorname*{int}H$. Then $\operatorname*{WMin}(F,H)\subseteq
\operatorname*{Min}(F,D^K)$.
\end{proposition}

\begin{proof}
Assume that $H$ is a convex cone with $\mathbb{R}_+^\ell\setminus \{ 0 \}
\subseteq \operatorname*{int}H$.\\
$\mathbb{R}_+^\ell\setminus \{ 0 \} \subseteq \operatorname*{int}H$. $\Rightarrow \forall
i\in \{1,\ldots,\ell\}
\quad\exists t_i>0:$
\[ y^i \mbox{ with }
y_j^i:=\left\{
\begin{array}{c@{\mbox{ if }}l}
1 & j=i,\\
-t_i & j\neq i,
\end{array}
\right. \]
is an element of $H$.\\
Assume $K>\max\limits_{i\in \{1,\ldots,\ell\} }\, \frac{1}{t_i}\quad$ and $y^0\in
D^K$.\\
$\Rightarrow \exists i\in \{ 1,\ldots,\ell\}:\quad y_i^0>0$ and
$y_i^0+Ky_j^0>0\quad
\forall j\in \{1,\ldots,\ell\} \setminus \{ i \}$.\\
$\bar{y}:=\frac{1}{y_i^0}y^0$. $\Rightarrow
\bar{y}_j=\frac{1}{y_i^0}y^0_j>-\frac{1}{K}>-t_i
\quad
\forall j\in \{1,\ldots,\ell\} \setminus \{ i \}$.\\
$\Rightarrow \bar{y}\in y^i+(\mathbb{R}_+^\ell\setminus \{ 0 \})\subseteq
H+\operatorname*{int}H\subseteq \operatorname*{int}H$.
$\Rightarrow$ since $H$ is a cone: $y^0\in \operatorname*{int}H$.\\
Hence $D^K\subseteq \operatorname*{int}H$. $\Rightarrow
\operatorname*{WMin}(F,H)\subseteq
\operatorname*{Min}(F,D^K)$.
\end{proof}

Proper efficiency according to Geoffrion is also related to efficiency w.r.t. the convex cones
\[ C^p:=\{ y\in \mathbb{R}^\ell  \mid  py_i+\sum\limits_{{j=1 \atop j\neq
i}}^\ell y_j\geq 0\quad\forall i\in \{1,\ldots ,\ell\}\}
\mbox{ with } p\in \mathbb{R}_>. \]

\begin{lemma}\label{hab-l522}
\begin{itemize}
\item[]
\item[\rm (a)] $1\leq p<\bar{p} \Rightarrow C^{\bar{p}}\setminus\{0\}
\subseteq int\,C^p$.
\item[\rm (b)] $0<p<1 \Rightarrow \exists \bar{p}>1,\tilde{p}>1:\quad
C^{\tilde{p}}\subseteq C^p \subseteq C^{\bar{p}}$.
\item[\rm (c)] $\forall \bar{p}>0 \,:\; 
\mathbb{R}_+^\ell =\bigcap\limits_{p\geq\bar{p}}C^p, \; \mathbb{R}_+^\ell\setminus\{0\}=\bigcap\limits_{p\geq\bar{p}}
int\,C^p$.
\item[\rm (d)] $\forall p\in\mathbb{R}_> \;\exists K\in\mathbb{R}_> \;\forall \tilde{K}\geq K
\,:\; D^{\tilde{K}}\subseteq C^p$.
\item[\rm (e)] $\forall K\in\mathbb{R}_> \,:\; C^p\setminus\{0\}\subseteq D^K$ for $p=\ell K$.
\end{itemize}
\end{lemma}

\begin{proof}
\begin{itemize}
\item[]
\item[\rm (a)] Assume $ 1\leq p<\bar{p}$, $ y\in
C^{\bar{p}}\setminus\{0\}$.
$\Rightarrow\quad \bar{p}y_i+\sum\limits_{{j=1 \atop j\neq
i}}^\ell y_j\geq 0\quad\forall i \in \{1,\ldots ,\ell\}$.
Suppose $\sum\limits_{j=1}^\ell y_j\leq 0$.
$\Rightarrow 0\leq\bar{p}y_i+\sum\limits_{{j=1 \atop j\neq
i}}^\ell y_j=(\bar{p}-1)y_i+\sum\limits_{j=1}^\ell y_j\le (\bar{p}-1)y_i
\quad\forall i \in \{1,\ldots ,\ell\}$.
$\Rightarrow y_i\geq 0\quad\forall i \in \{1,\ldots ,\ell\}$, a contradiction to the supposition because of $y\neq 0$.
$\Rightarrow \sum\limits_{j=1}^\ell y_j>0$.
If $y_i\geq 0$,
we get
$py_i+\sum\limits_{{j=1 \atop j\neq i}}^\ell y_j\geq
\sum\limits_{j=1}^\ell y_j>0$.
For $y_i<0$, we get
$py_i+\sum\limits_{{j=1 \atop j\neq
i}}^\ell y_j>\bar{p}y_i+\sum\limits_{{j=1 \atop j\neq
i}}^\ell y_j\geq 0$.
$\Rightarrow y\in \operatorname*{int}C^p$.
\item[\rm (b)] Assume $p_0\in \mathbb{R}_>$, $y\in C^{p_0}$.
$\Rightarrow p_0y_i+ \sum\limits_{{j=1 \atop j\neq i}}^\ell y_j\geq 0
\quad\forall i\in \{1,\ldots ,\ell\}$.\\
Adding these inequalities for $i\neq m\in \{1,\ldots ,\ell\}$
yields: \\
$p_0\sum\limits_{{j=1 \atop j\neq
m}}^\ell y_j+(\ell -1)y_m+(\ell -2)\sum\limits_{{j=1 \atop j\neq m}}^\ell y_j
\geq 0\quad\forall m\in \{1,\ldots ,\ell\}$.\\
$\Rightarrow p_{00}y_m+\sum\limits_{{j=1 \atop j\neq m}}^\ell y_j\geq
0\quad\forall m\in \{1,\ldots ,\ell\}\qquad$
with $p_{00}:=\frac{\ell -1}{p_0+(\ell -2)}$.\\
$\Rightarrow y\in C^{p_{00}}$, thus $C^{p_0}\subseteq C^{p_{00}}$.\\
Consider first $p_0:=p<1$. $\Rightarrow \bar{p}:=p_{00}>1$.\\
Consider now $\tilde{p}:=p_0:=\frac{\ell -1-(\ell -2)p}{p}>1$.
$\Rightarrow p_{00}=p$.
\item[\rm (c)] 
Take any $y\notin \mathbb{R}_+^\ell$. $\Rightarrow \exists
i\in\{1,\ldots,\ell\} \, :\; y_i<0$. For $p> \max(\bar{p},-\frac{1}{y_i}\sum\limits_{{j=1
\atop j\neq i}}^\ell y_j)$, we have
$py_i+\sum\limits_{{j=1 \atop j\neq i}}^\ell y_j < 0$ and
hence $y\notin C^p$. Thus, $\bigcap\limits_{p\geq\bar{p}}C^p\subseteq \mathbb{R}_+^\ell$.
Since $\mathbb{R}_+^\ell\subseteq C^p$ for each $p\in\mathbb{R}_>$, we get $\bigcap\limits_{p\geq\bar{p}}C^p= \mathbb{R}_+^\ell$.\\
Since $\operatorname*{int}C^p\subseteq C^p\setminus\{0\}$ for each $p\in\mathbb{R}_>$, we get $\bigcap\limits_{p\geq\bar{p}}
int\,C^p\subseteq \mathbb{R}_+^\ell\setminus\{0\}$. Because $\mathbb{R}_+^\ell\setminus\{0\}\subseteq \operatorname*{int}C^p$ holds for each $p\in\mathbb{R}_>$, we get the assertion.
\item[\rm (d)] Assume $p>0$. Choose $K>\max(\frac{\ell -1}{p},\,
p+\ell -2)$. Then,
\begin{equation}
\label{522}
p+(\ell -1)(-\frac{1}{K})>0 \mbox{ and }
p(-\frac{1}{K})+1+(\ell -2)(-\frac{1}{K})>0.
\end{equation}
Consider an arbitrary $y\in D^K$.\\
$\Rightarrow \exists i\in \{1,\ldots,\ell\}:\quad y_i>0$ and
$y_j>-\frac{1}{K}y_i
\quad\forall j\in \{1,\ldots,\ell\} \setminus \{ i \}$.\\
$\Rightarrow$ with (\ref{522}):
$py_i+\sum\limits_{{j=1 \atop j\neq
i}}^\ell y_j>(p+(\ell -1)(-\frac{1}{K}))y_i>0$
and $\forall n\in \{1,\ldots,\ell\}\setminus\{i\}:\\
py_n+\sum\limits_{{j=1
\atop j\neq n}}^\ell y_j=
py_n+y_i+\sum\limits_{{j=1 \atop j\notin
\{n,i\} }}^\ell y_j>(p(-\frac{1}{K})+1+(\ell -2)(-\frac{1}{K}))y_i>0$.\\
$\Rightarrow \forall n\in \{1,\ldots,\ell\}:\quad
py_n+\sum\limits_{{j=1 \atop j\neq n}}^\ell y_j>0$.
$\Rightarrow y\in C^p$.\\
Hence, $D^K\subseteq C^p$. By Proposition \ref{hab-s521}, $D^{\tilde{K}}\subseteq D^K$ for all $\tilde{K}\geq K$.
Thus, $D^{\tilde{K}}\subseteq C^p$ for all $\tilde{K}\geq K$.
\item[\rm (e)] Consider $p=\ell K$.
Take any $y\in C^p\setminus \{ 0 \}$.
$\Rightarrow py_i+\sum\limits_{{j=1 \atop j\neq i}}^\ell y_j\geq
0\quad\forall i\in \{1,\ldots,\ell\}$.
$\Rightarrow y_n:=\max\limits_{i\in \{1,\ldots,\ell\} }\, y_i\, >0$.
$\Rightarrow \sum\limits_{{j=1 \atop j\neq i}}^\ell y_j\leq (\ell -1)y_n<\ell y_n
\quad\forall i\in \{1,\ldots,\ell\} \setminus \{ n \}$.\\
$\Rightarrow y_i+\frac{\ell }{p}\, y_n>
y_i+\frac{1}{p}\sum\limits_{{j=1 \atop j\neq i}}^\ell y_j\geq 0
\quad\forall i\in \{1,\ldots,\ell\} \setminus \{ n \}$.\\
$\Rightarrow y_n+Ky_i=y_n+\frac{p}{\ell }\, y_i> 0
\quad\forall i\in \{1,\ldots,\ell\} \setminus \{ n \}$.
$\Rightarrow y\in D^K$.\\
Thus, $C^p\setminus \{ 0 \}\subseteq D^K$.
\end{itemize}
\end{proof}

This implies by Proposition \ref{hab-s523}:

\begin{corollary}\label{c-pol-cone}
For each convex cone $H$ with $\mathbb{R}_+^\ell\setminus \{ 0 \}\subseteq \operatorname*{int}H$, there exists some polyhedral cone $C=C^p$, $p\in\mathbb{R}_>$, with $\mathbb{R}_+^\ell\setminus \{ 0 \}\subseteq \operatorname*{int}C$ and $C\setminus \{ 0 \}\subseteq \operatorname*{int}H$
. Then $\operatorname*{WMin}(F,H)\subseteq
\operatorname*{Min}(F,C)$.
\end{corollary}

\begin{proposition}\label{hab-s522}
$\operatorname*{GMin}(F)=\bigcup\limits_{p>0}\operatorname*{Min}(F,C^p)=\bigcup\limits_{p\geq
\bar{p}}\operatorname*{Min}(F,C^p)  \quad
\forall \bar{p}\geq 1$.\\
Here, $\operatorname*{Min}(F,C^{\bar{p}})\subseteq \operatorname*{Min}(F,C^{p})
\quad\forall
1\leq \bar{p}\leq p\qquad $
and \\
$\forall p\in(0,1):\quad\exists \bar{p}>1,\, \tilde{p}>1:\quad
\operatorname*{Min}(F,C^{\bar{p}})\subseteq \operatorname*{Min}(F,C^{p})\subseteq
\operatorname*{Min}(F,C^{\tilde{p}})$.
\end{proposition}

\begin{proof}
\begin{itemize}
\item[]
\item[\rm (a)] The inclusions between the efficient point sets in the second part of the proposition result from Lemma \ref{hab-l522}.
\item[\rm (b)] Assume $p>0$. By Lemma \ref{hab-l522}, there exists some $K>0$ with
$\operatorname*{Min}(F,C^p)\subseteq \operatorname*{Min}(F,D^K)$. Proposition \ref{hab-s521}
implies $\operatorname*{Min}(F,C^p)\subseteq \operatorname*{GMin}(F)$.
\item[\rm (c)] Assume $\bar{y}\in \operatorname*{GMin}(F)$, $\bar{p}\geq 1$.\\
$\Rightarrow \exists K>\bar{p}:\quad \bar{y}\in \operatorname*{Min}(F,D^K)$.
$\quad p:=\ell K>\bar{p}$.
By Lemma \ref{hab-l522}, $C^p\setminus \{ 0 \}\subseteq D^K$.
$\Rightarrow \operatorname*{Min}(F,D^K)\subseteq \operatorname*{Min}(F,C^p)$.
$\Rightarrow \bar{y}\in \operatorname*{Min}(F,C^p)$. 
\end{itemize}
\end{proof}

\begin{corollary}\label{c-geoff-poly}
The following statements are equivalent to each other:
\begin{itemize}
\item[\rm (a)] $y\in\operatorname*{GMin}(F)$.
\item[\rm (b)] $y\in \operatorname*{Min}(F,H)$ for some polyhedral cone $H$ with $\mathbb{R}^\ell _+\setminus\{0\}\subseteq\operatorname*{int}H$.
\item[\rm (c)] $y\in\operatorname*{WMin}(F,H)$ for some polyhedral cone $H$ with $\mathbb{R}^\ell _+\setminus\{0\}\subseteq \operatorname*{int}H$.
\end{itemize}
\end{corollary}
\begin{proof}
(a) implies (b) by Proposition \ref{hab-s522}. (b) implies (c) because of 
$\operatorname*{Min}(F,H)\subseteq \operatorname*{WMin}(F,H)$. (c) implies $y\in\operatorname*{GMin}(F)$ by Proposition \ref{p-Geoff-cones}.
\end{proof}

Proposition \ref{hab-s522} implies the following statement by Podinovskij and Nogin \cite[Theorem 2.1.15.]{pono82}.

\begin{corollary}\label{hab-c521}
$\operatorname*{GMin}(F)=\bigcup\limits_{\epsilon \in
(0,\frac{1}{\ell })}\operatorname*{Min}(F,\Lambda_{\epsilon })$
with \\
$\Lambda_{\epsilon }:=\{ y\in \mathbb{R}^\ell  \mid  (1-(\ell -1)\epsilon
)y_i+\epsilon
\sum\limits_{{j=1 \atop j\neq i}}^\ell y_j\geq 0 \quad\forall
i\in \{1,\ldots,\ell\}\}$.
\end{corollary}

\begin{proof}
$\Lambda_{\epsilon }=C^p$ with $p=\frac{1-(\ell -1)\epsilon }{\epsilon
}$ and
$\epsilon =\frac{1}{p+\ell -1}$.
$\quad \epsilon \in (0,\frac{1}{\ell })\iff p>1.$
\end{proof}

Proper efficiency according to Geoffrion can also be characterized as efficiency w.r.t. the convex cones
\[ C(s):=\{ y\in \mathbb{R}^\ell  \mid  y_i+s^Ty\geq 0\quad\forall i\in \{1,\ldots,\ell\}\}
\mbox{  with } s\in \operatorname*{int}\mathbb{R}_+^\ell . \]

\begin{lemma}\label{zu-hab-s524}
Assume $s\in \operatorname*{int}\mathbb{R}_+^\ell$.
\begin{itemize}
\item[\rm (a)] $\forall\,m\in\mathbb{R}_>\,:\quad \mathbb{R}_+^\ell\subseteq C(\frac{s}{m}),\quad \mathbb{R}_+^\ell\setminus \{ 0 \}\subseteq \operatorname*{int} C(\frac{s}{m})$.
\item[\rm (b)] $\forall\,m>1 \,:\quad C(\frac{s}{m})\setminus\{0\}\subset D^K$ \quad for \quad $K=\frac{m-1}{\sum\limits_{i=1}^\ell s_i}$.
\end{itemize}
\end{lemma}

\begin{proof}
\begin{itemize}
\item[]
\item[\rm (a)] is obvious.
\item[\rm (ii)] Take any $m>1$. $K:=\frac{m-1}{\sum\limits_{i=1}^\ell s_i}$.
Consider an arbitrary $y\in C(\frac{s}{m})\setminus \{ 0 \}$.\\
$\Rightarrow y\neq 0$ and $\frac{s^Ty}{m}+y_i\geq 0\quad\forall
i\in \{1,\ldots,\ell\}$.
$\Rightarrow y_n:=\max\limits_{i\in \{1,\ldots,\ell\} }\, y_i\, >0$.
$\Rightarrow \frac{s^Ty}{m}\leq \frac{1}{m}\sum\limits_{i=1}^\ell s_i\,
y_n<\frac{1}{m-1}\sum\limits_{i=1}^\ell s_i\,
y_n$.
$\Rightarrow \forall i\in \{1,\ldots,\ell \} :\quad
y_i+\frac{1}{m-1}\sum\limits_{i=1}^\ell s_i\, y_n>
y_i+\frac{s^Ty}{m}\geq 0$.\\
$\Rightarrow \forall i\in \{1,\ldots,\ell \} :\quad
y_n+\frac{m-1}{\sum\limits_{i=1}^\ell s_i} y_i> 0$.
$\Rightarrow y\in D^K$.
\end{itemize}
\end{proof}

\begin{proposition}\label{hab-s524}
$\forall s\in \operatorname*{int}\mathbb{R}_+^\ell :\quad
\operatorname*{GMin}(F)=\bigcup\limits_{m>0}\,
\operatorname*{Min}(F,C(\frac{s}{m}))$,\\
where $\forall y\in \operatorname*{GMin}(F)\quad\exists
m_0>1\quad\forall
m\geq m_0:
\quad y\in \operatorname*{Min}(F,C(\frac{s}{m}))$.
\end{proposition}

\begin{proof}
\begin{itemize}
\item[]
\item[\rm (i)] $\mathbb{R}_+^\ell\setminus \{ 0 \}\subseteq \operatorname*{int}
C(\frac{s}{m})\quad\forall s\in \operatorname*{int}\mathbb{R}_+^\ell ,\, m>0$.
Thus, we deduce from  Proposition \ref{p-Geoff-cones}:\\ 
$\bigcup\limits_{m>0}\operatorname*{Min}(F,C(\frac{s}{m}))\subseteq
\operatorname*{GMin}(F)$.
\item[\rm (ii)] Take any $y^0\in \operatorname*{GMin}(F)$, $s\in \operatorname*{int}\mathbb{R}_+^\ell$.\\
Because of Proposition \ref{hab-s521}, there exists some $K_0>0$ such that\\
$y^0\in \operatorname*{Min}(F,D^K)\quad\forall K\geq K_0$.\\
$m_0:=1+K_0\sum\limits_{i=1}^\ell s_i>1$.
Take any $m\geq m_0$. $K:=\frac{m-1}{\sum\limits_{i=1}^\ell s_i}\geq K_0$, and $C(\frac{s}{m})\setminus \{ 0 \}\subseteq D^K$ by Lemma \ref{zu-hab-s524}. 
Hence, $\operatorname*{Min}(F,D^K)\subseteq \operatorname*{Min}(F,C(\frac{s}{m}))$.\\
$\Rightarrow y^0\in \operatorname*{Min}(F,C(\frac{s}{m}))\quad\forall m\geq m_0$.
\end{itemize}
\end{proof}

Efficiency and proper efficiency according to Geoffrion coincide for linear vector optimization problems
(see, e.g., \cite{fock73}). Consequently, for these problems the following statement of Helbig  \cite{helb90e}
is stronger than the above proposition.\\
For any $s\in \operatorname*{int}\mathbb{R}_+^\ell$, there exists some $m_0\in \mathbb{N}$ such that for all $m\geq m_0$ \,: 
\[\operatorname*{WMin}(F,C(\frac{s}{m}))=\operatorname*{Min}(F,C(\frac{s}{m}))=\operatorname*{Min}(F).\]
Here, the constant $m_0$ can be chosen independently from the considered efficient element. That this is in general not the case for other than linear vector optimization problems can be illustrated by a simple example:

\begin{example}\label{hab-e521}
Consider
$F=\{ (y_1,y_2)^T\in \mathbb{R}^2 \mid  -1\leq y_1\leq 0,\, -y_1\leq y_2\leq 1\}
\cup \{ (y_1,y_2)^T\in \mathbb{R}^2 \mid  0\leq y_1\leq 1,\, -1\leq y_2\leq 1 \}$.
$\operatorname*{GMin}(F)=\{ (0,-1)^T\} \cup \{ (y_1,y_2)^T\in \mathbb{R}^2 \mid  -1\leq
y_1<0,\, y_2=-y_1\}$,
but there does not exist any convex cone $C$ with $\mathbb{R}_+^\ell\setminus \{
0\} \subseteq \operatorname*{int}C$ such that
$\operatorname*{Min}(F,C)=\operatorname*{GMin}(F)$.
\end{example}

\begin{lemma}\label{zu-hab-s525}
Assume $w\in \operatorname*{int}\mathbb{R}_+^\ell$. Define
\[ C_w(\epsilon ):=\{ y\in \mathbb{R}^\ell  \mid  w_iy_i+\epsilon\sum\limits_{j=1
}^\ell w_jy_j\geq 0\quad\forall i\in \{1,\ldots,\ell\}\}\]
for all $\epsilon\in \mathbb{R}_>$.
Then:
\begin{itemize}
\item[\rm (a)] $\forall\,\epsilon\in\mathbb{R}_>\,:\quad \mathbb{R}_+^\ell\subseteq C_w(\epsilon),\quad \mathbb{R}_+^\ell\setminus \{ 0 \}\subseteq \operatorname*{int} C_w(\epsilon)$.
\item[\rm (b)] $\forall\,\epsilon\in\mathbb{R}_> \,:\quad C_w(\epsilon)\setminus\{0\}\subset D^K$ \quad for \quad $K=\frac{\min\limits_{i\in \{1,\ldots,\ell\} }w_i}{2\epsilon\sum\limits_{j=1}^\ell w_j}$.
\end{itemize}
\end{lemma}

\begin{proof}
\begin{itemize}
\item[]
\item[\rm (a)] is obvious.
\item[\rm (b)] Assume $\epsilon\in\mathbb{R}_>$.
$K:=\frac{\min\limits_{i\in \{1,\ldots,\ell\} }w_i}{2\epsilon\sum\limits_{j=1}^\ell w_j}$.
Take any $y\in C_w(\epsilon)\setminus \{ 0 \}$.
$\Rightarrow y_n:=\max\limits_{i\in \{1,\ldots,\ell\} }\, y_i\, >0$.
$\Rightarrow \frac{\epsilon}{w_i}\sum\limits_{j=1}^\ell w_jy_j
\leq \frac{\epsilon}{w_i}\sum\limits_{j=1}^\ell w_jy_n<
\frac{2\epsilon}{w_i}\sum\limits_{j=1}^\ell w_jy_n\leq \frac{1}{K}\, y_n$
for all $i\in \{1,\ldots,\ell\}$.
$\Rightarrow \forall i\in \{1,\ldots,\ell\} :\quad
y_i+\frac{1}{K}\, y_n>
y_i+\frac{\epsilon}{w_i}\sum\limits_{j=1}^\ell w_j\,y_j \geq 0$.
$\Rightarrow \forall i\in \{1,\ldots,\ell\} :\quad
y_n+K y_i> 0$.
$\Rightarrow y\in D^K$.
\end{itemize}
\end{proof}

\begin{proposition}\label{hab-s525}
Assume $w\in \operatorname*{int}\mathbb{R}_+^\ell$.
Then:\\
$\operatorname*{GMin}(F)=\bigcup\limits_{\epsilon >0}\,
\operatorname*{Min}(F,C_w(\epsilon))$,
where $\forall y\in \operatorname*{GMin}(F)\quad\exists
\epsilon_0>0\quad\forall
\epsilon\in (0,\epsilon_0]:
\quad y\in \operatorname*{Min}(F,C_w(\epsilon))$.
\end{proposition}

\begin{proof}
\begin{itemize}
\item[]
\item[\rm (i)] $\mathbb{R}_+^\ell\setminus \{ 0 \}\subseteq \operatorname*{int}
C_w(\epsilon)\quad\forall \epsilon>0.$
Hence, Proposition \ref{p-Geoff-cones} implies\\
$\bigcup\limits_{\epsilon>0}\operatorname*{Min}(F,C_w(\epsilon))\subseteq
\operatorname*{GMin}(F)$.
\item[\rm (ii)] Assume $y^0\in \operatorname*{GMin}(F)$.\\
Because of Proposition \ref{hab-s521}, there exists some $K_0>0$ such that\\
$y^0\in \operatorname*{Min}(F,D^K)\quad\forall K\geq K_0$.\\
$\epsilon_0:=\frac{\min\limits_{i\in \{1,\ldots,\ell\} }w_i}{2K_0\sum\limits_{j=1}^\ell w_j}>0$.
Consider $\epsilon\in (0,\epsilon_0]$.
$K:=\frac{\min\limits_{i\in \{1,\ldots,\ell\} }w_i}{2\epsilon\sum\limits_{j=1}^\ell w_j}\geq K_0$.\\
$C_w(\epsilon)\setminus \{ 0 \}\subseteq D^K$ by Lemma \ref{zu-hab-s525}. Thus,
$\operatorname*{Min}(F,D^K)\subseteq \operatorname*{Min}(F,C_w(\epsilon))$.
$\Rightarrow y^0\in \operatorname*{Min}(F,C_w(\epsilon))\quad\forall
\epsilon\in (0,\epsilon_0]$.
\end{itemize}
\end{proof}

Analogously, one proves:

\begin{proposition}\label{hab-s526}
Suppose $w\in \operatorname*{int}\mathbb{R}_+^\ell$. Define
\[ C_{w,\epsilon}:=\{ y\in \mathbb{R}^\ell  \mid  w_iy_i+\epsilon\sum\limits_{j=1
}^\ell y_j\geq 0\quad\forall i\in \{1,\ldots,\ell\}\}\]
for all $\epsilon\in \mathbb{R}_>$.
Then, $\operatorname*{GMin}(F)=\bigcup\limits_{\epsilon >0}\,
\operatorname*{Min}(F,C_{w,\epsilon})$,\\
where $\forall y\in \operatorname*{GMin}(F)\quad\exists
\epsilon_0>0\quad\forall
\epsilon\in (0,\epsilon_0]:
\quad y\in \operatorname*{Min}(F,C_{w,\epsilon})$.
\end{proposition}

\begin{proof}
Follow the proof of Proposition \ref{hab-s525}, but replace  $C_w(\epsilon)$ by
$C_{w,\epsilon}$. There,
$\epsilon_0:=\frac{\min\limits_{i\in \{1,\ldots,\ell\} }w_i}{2\ell\,K_0}$,
$K:=\frac{\min\limits_{i\in \{1,\ldots,\ell\} }w_i}{2\ell\epsilon}$, and the term
$\sum\limits_{j=1}^\ell w_j$ should be replaced  by $\ell$ everywhere where it is not a part of
$\sum\limits_{j=1}^\ell w_jy_j$.
\end{proof}

Let us point out that the polyhedral cones $C^p$, $C(s)$,
$C_w(\epsilon)$ and $C_{w,\epsilon}$ have the form
\[\{ y\in \mathbb{R}^\ell  \mid  p_iy_i+\sum\limits_{{j=1 \atop j\neq i}}^\ell s_jy_j\geq 0
\quad\forall i\in \{1,\ldots,\ell\}\}\]
with $s\in \operatorname*{int}\mathbb{R}_+^\ell$, $p\in \operatorname*{int}\mathbb{R}_+^\ell$.

Our results imply statements about the existence of properly efficient elements. 

\begin{proposition}\label{hab-s527}
\begin{itemize}
\item[]
\item[\rm (a)] Assume $u\in\mathbb{R}^\ell$,
$K_0$, $\bar{p}$, $\epsilon_0\in\mathbb{R}_>$ and
$s$, $w \in \operatorname*{int}\mathbb{R}_+^\ell$.
The following statements are equivalent to each other:
\begin{itemize}
\item[\rm (i)] $\exists$ a convex cone $H$ with
$\mathbb{R}_+^\ell\setminus\{0\}\subseteq
\operatorname*{int}H \, :\;  (F-u)\cap (-\operatorname*{int}H)=\emptyset$,
\item[\rm (ii)] $\exists$ a polyhedral cone $H$ with
$\mathbb{R}_+^\ell\setminus\{0\}\subseteq
\operatorname*{int}H \, :\;  (F-u)\cap (-\operatorname*{int}H)=\emptyset$,
\item[\rm (iii)] $\exists\, K\geq K_0 \, :\;  (F-u)\cap (-D^K)=\emptyset$,
\item[\rm (iv)] $\exists\, p\geq \bar{p} \, :\;  (F-u)\cap (-C^p)\subseteq\{0\}$,
\item[\rm (v)] $\exists\, m>0 \, :\; 
(F-u)\cap (-C(\frac{s}{m}))\subseteq\{0\}$,
\item[\rm (vi)] $\exists\, \epsilon\in (0,\epsilon_0] \, :\;  (F-u)\cap (-C_w(\epsilon))
\subseteq \{0\}$.
\end{itemize}
\item[\rm (b)] $\operatorname*{GMin}(F)\neq\emptyset$ if and only if (i) holds for some
$u\in F$.
\item[\rm (c)] If $F$ is nonempty and closed, then $\operatorname*{GMin}(F)\neq\emptyset$ if and only if (ii) holds for some
$u\in \mathbb{R}^\ell$.
\end{itemize}
\end{proposition}

\begin{proof}
\begin{itemize}
\item[]
\item[\rm (a)] The equivalence of (iii) and (iv) follows from Lemma \ref{hab-l522}.\\
(i) implies (iii) because of Proposition \ref{hab-s523} and Proposition \ref{hab-s521}.
(iii) implies (v) because of Lemma \ref{zu-hab-s524}.
(v) implies (ii), since $C(\frac{s}{m})$ is a polyhedral cone with
$\mathbb{R}_+^\ell\setminus\{0\}\subseteq \operatorname*{int}C(\frac{s}{m})$. (ii) yields (i).\\
(vi) implies (ii). (iii) implies $(F-u)\cap (-D^{\tilde{K}} )=\emptyset$ for each $\tilde{K} >K$ by Proposition \ref{hab-s521}, and thus (vi) because of Lemma \ref{zu-hab-s525}.
\item[\rm (b)] results from Proposition \ref{p-Geoff-cones}.
\item[\rm (c)] Assume that $F\neq\emptyset$ is closed and that (ii) is fulfilled for some $u\in\mathbb{R}^\ell$.
By Corollary \ref{c-pol-cone}, there exists a polyhedral cone $T$ with
$\mathbb{R}_+^\ell\setminus\{0\}\subseteq
\operatorname*{int}T$and $T\setminus\{0\}\subseteq \operatorname*{int}H$. Choose some $\bar{y}\in
F$.
$\tilde{F}:=F\cap(\bar{y}-T)$ is compact by \cite[Lemma 3]{weid93}.
For an arbitrary $k\in \operatorname*{int}T$, the functional
$\varphi _{-T,k}$ is continuous and finite-valued by \cite[Prop. 4.1]{Wei17a}.
Hence, it attains a minimum $t$ on $\tilde{F}$.
$\Rightarrow \tilde{F}\cap(-int\,T+tk)=\emptyset$ by \cite[Theorem 3.1]{Wei17a}, and $\exists
y^0\in\tilde{F} \, :\; 
y^0\in -T+tk$.\\
Suppose that $(F-y^0)\cap (-\operatorname*{int}T)\neq\emptyset$.\\
$\Rightarrow\exists y\in F\cap(y^0-\operatorname*{int}T)$.
$\Rightarrow y\in -T+tk-\operatorname*{int}T\subseteq -\operatorname*{int}T+tk$ since $y^0\in
-T+tk$, and
$y\in \bar{y}-T-int\,T\subseteq \bar{y}-T$ since $y^0\in \bar{y}-T$.
$\Rightarrow y\in\tilde{F}\cap(-\operatorname*{int}T+tk)$, which is impossible.
Thus, the supposition $(F-y^0)\cap (-\operatorname*{int}T)\neq\emptyset$ delivers a contradiction. Consequently, (ii) holds for $u:=y^0\in
F$. This and (b) imply the assertion.
\end{itemize}
\end{proof}

For the set $F:=\mathbb{R}_+^2\setminus\{0\}$, which is not closed, we have
$\operatorname*{GMin}(F)=\emptyset$, though the condition (i) is fulfilled with
$u=(0,0)^T$.

In part (c) of the following theorem, we will use the assumption

\smallskip
\begin{tabular}{ll}
(Sp-ex): & There exist a polyhedral cone $S\subseteq \mathbb{R}^\ell$ and some $u\in\mathbb{R}^\ell$\\
& such that $\mathbb{R}_+^\ell\setminus\{0\}\subseteq \operatorname*{int}S$ and $(F-u)\cap (-\operatorname*{int}S)=\emptyset$.
\end{tabular}\\
\smallskip
By Proposition \ref{hab-s527}, (Sp-ex) is equivalent to
$\operatorname*{GMin}(F)\neq\emptyset$ if $F$ is nonempty and closed.

\begin{theorem}\label{hab-t521}
\begin{itemize}
\item[]
\item[\rm (a)] One has:
\begin{eqnarray*}
\qquad\quad\operatorname*{GMin}(F) & = & \{y^0\in F \mid  \exists \varphi :\mathbb{R}^{\ell}\to \mathbb{R} \mbox{
strictly }
\mathbb{R}_+^\ell\mbox{--monotone, sublinear,}\\
 & & \mbox{continuous}:\, \varphi (y^0-y^0)=\operatorname*{min}_{y\in F}\varphi (y-y^0)=0\}\\
  & = & \{y^0\in F \mid  \exists \varphi :\mathbb{R}^{\ell}\to \mathbb{R} \mbox{ strictly }
\mathbb{R}_+^\ell\mbox{--monotone, sublinear,}\\
 & & \mbox{continuous}:\, \forall y\in F\setminus \{ y^0\} :\; \varphi (y-y^0)> 0\}.
\end{eqnarray*}
In detail,:
\begin{itemize}
\item[\rm (i)] $y^0\in \operatorname*{GMin}(F)$ if and only if there exists some non-trivial closed convex cone $H$ with $\mathbb{R}^{\ell}_+ \setminus \{ 0 \}\subseteq \operatorname*{int}H$ such that $\varphi_{-H,k}(y^0-y^0)=\operatorname*{min}_{y\in F}\varphi_{-H,k}(y-y^0)=0$, where $k\in \operatorname*{int}H$ can be chosen arbitrarily.
$\varphi_{-H,k}$ is finite-valued, continuous, sublinear and strictly $\mathbb{R}^{\ell}_+$--monotone.
\item[\rm (ii)] If $y^0\in \operatorname*{GMin}(F)$, then there exists some $K\in\mathbb{R}_>$ such that $y^0\in \operatorname*{Min}(F,D^K)$. For each $k\in\operatorname*{int}
\mathbb{R}_+^\ell$, $\varphi :=\sum\limits_{i=1}^\ell \varphi _{-\operatorname*{cl}D^{i,K},k}$ is finite-valued, strictly 
$\mathbb{R}_+^\ell $--monotone, sublinear, continuous, and $\varphi (y-y^0)> 0 $ for all $y\in F\setminus \{ y^0\}$.
\end{itemize}
\item[\rm (b)] Furthermore,
\begin{eqnarray*}
\qquad\quad\operatorname*{GMin}(F) & \subseteq & \{y^0\in F \mid  \exists \varphi :\mathbb{R}^{\ell}\to \mathbb{R}
\mbox{ strictly }
\mathbb{R}_+^\ell \mbox{--monotone, strictly}\\
 & & \mbox{convex, continuous}:\;\forall y\in F\setminus \{ y^0\} :\; \varphi (y)>\varphi (y^0) \}.
\end{eqnarray*}
\item[\rm (c)] Assume {\rm (Sp-ex)}. Then
\begin{eqnarray*}
\qquad\quad\operatorname*{GMin}(F) & = & \{y^0\in F \mid  \exists \varphi :\mathbb{R}^{\ell}\to \mathbb{R} \mbox{
strictly }
\mathbb{R}_+^\ell\mbox{--monotone, strictly}\\
 & & \mbox{convex, continuous}:\;\forall y\in F\setminus \{ y^0\} :\; \varphi (y)>\varphi (y^0) \}\\
 & = & \{y^0\in F \mid  \exists \varphi :\mathbb{R}^{\ell}\to \mathbb{R} \mbox{ strictly }
\mathbb{R}_+^\ell\mbox{--monotone, convex,}\\
 & & \mbox{continuous}:\,\varphi (y^0)=\min\limits_{y\in F}\, \varphi (y) \}.
\end{eqnarray*}
\end{itemize}
\end{theorem}

\begin{proof}
\begin{itemize}
\item[]
\item[\rm (a)] The first equation and (i) result from Corollary \ref{Henig_minima}. We will now prove (ii), what implies the second equation.\\
Assume $y^0\in \operatorname*{GMin}(F)$, $k\in \operatorname*{int}
\mathbb{R}_+^\ell$.\\
$\Rightarrow \exists K>0 \,:\;  y^0\in \operatorname*{Min}(F,D^K)$, thus
$y^0\in \operatorname*{Min}(F,D^{i,K})\quad \forall i\in \{ 1,\ldots,\ell\}$.
Because of \cite[Prop. 4.1]{Wei17a}, $\varphi _{-\operatorname*{cl}D^{i,K},k}$ is finite-valued, continuous and sublinear for all $i\in \{ 1,\ldots,\ell\}$.
$(F-y^0)\cap (-D^{i,K})=\emptyset$ implies, by  \cite[Theorem 3.1]{Wei17a}, $\varphi _{-\operatorname*{cl}D^{i,K},k}(y-y^0)\geq 0\quad \forall y\in F, i\in \{ 1,\ldots,\ell\}$.
$\Rightarrow \varphi :=\sum\limits_{i=1}^\ell \varphi _{-\operatorname*{cl}D^{i,K},k}$ is finite-valued, continuous, sublinear, and
$\varphi (y-y^0)\geq 0\;\forall \, y\in F$.\\
If $\varphi (y-y^0)=0$ for some $y\in F$, then
$\varphi _{-\operatorname*{cl}D^{i,K},k}(y-y^0)=0\quad\forall
i\in \{ 1,\ldots,\ell\}$,\\
hence $y-y^0\in -\operatorname*{bd}D^{i,K}\quad\forall i\in \{ 1,\ldots,\ell\}$ by \cite[Theorem 3.1]{Wei17a}, and thus
$y=y^0$. \\
$\Rightarrow \varphi (y-y^0)>0\quad \forall y\in F\setminus \{ y^0 \}$.\\
Assume $y^2\in \mathbb{R}^\ell$, $y^1\in y^2+(\mathbb{R}_+^\ell\setminus \{ 0 \} )$.\\
$\operatorname*{cl}D^{i,K}+\mathbb{R}_+^\ell\subseteq \operatorname*{cl}D^{i,K}\quad\forall i\in \{ 1,\ldots,\ell\}$.\\
$\Rightarrow$ $\forall i\in \{ 1,\ldots,\ell\}:\;\varphi _{-\operatorname*{cl}D^{i,K},k}\quad \mathbb{R}_+^\ell$--monotone  by \cite[Theorem 2.16]{Wei17a}.\\
$\Rightarrow \varphi _{-\operatorname*{cl}D^{i,K},k}(y^1)\geq \varphi _{-\operatorname*{cl}D^{i,K},k}(y^2)\quad\forall
i\in \{ 1,\ldots,\ell\}$.\\
$\exists n\in \{ 1,\ldots,\ell\} :\quad y_n^1>y_n^2$.
$\quad r:=\varphi _{-\operatorname*{cl}D^{n,K},k}(y^2)$. $\Rightarrow y^2\in -\operatorname*{bd}D^{n,K}+rk$.\\
$\Rightarrow -y^2+rk\in \operatorname*{bd}D^{n,K}$. \\
$y_n^1>y_n^2$ and $y^1\geq y^2$ imply $(-y^1+rk)_n<(-y^2+rk)_n$
and \\
$(-y^1+rk)_n +K(-y^1+rk)_j<(-y^2+rk)_n+K(-y^2+rk)_j\quad\forall
j\in \{1,\ldots,\ell\}\setminus \{ n\} $.\\
Thus, $-y^1+rk\notin \operatorname*{cl}D^{n,K}$. Hence, $y^1\notin -\operatorname*{cl}
D^{n,K}+rk$.\\
$\Rightarrow \varphi _{-\operatorname*{cl}D^{n,K},k}(y^1)>r=\varphi _{-\operatorname*{cl}D^{n,K},k} (y^2)$.
$\Rightarrow \varphi (y^1)>\varphi (y^2)$.\\
Thus, $\varphi$ is strictly $\mathbb{R}_+^\ell$--monotone.
\item[\rm (b)] was proved by Soland \cite[Lemma 3]{sola79}. 
\item[\rm (c)] Assume $y^0\in F$, that $\varphi : \mathbb{R}^\ell \to \mathbb{R}$ is strictly $\mathbb{R}_+^\ell$--monotone,
continuous, convex, and $\varphi (y^0)=\min\limits_{y\in F}\, \varphi (y)$.\\
\cite[Theorem 3.3]{GerWei90} implies $H:=\{ y\in \mathbb{R}^\ell \mid  \varphi (y^0-y)<\varphi (y^0)\}$ is open and convex, $0\in \operatorname*{bd}H$, $\mathbb{R}_+^\ell\setminus
\{ 0 \} \subseteq H$ and, moreover, $\operatorname*{cl}H+(\mathbb{R}_+^\ell\setminus
\{ 0 \} )\subseteq H$, and
$y^0\in \operatorname*{Min}(F,H)$.\\
$\mathbb{R}_+^\ell\setminus\{0\}\subseteq \operatorname*{int}S$. Hence, by Corollary \ref{c-pol-cone}, there exists some
polyhedral cone $T$ with
$\mathbb{R}_+^\ell\setminus\{0\}\subseteq \operatorname*{int}T$ and $T\setminus\{0\}\subseteq \operatorname*{int}S$.
Consider $T=\{y\in\mathbb{R}^\ell  \mid  \sum\limits_{i=1}^\ell s_i^j y_i\geq
0\quad\forall j\in \{ 1,\ldots ,m\}\}$. \\
$\mathbb{R}_+^\ell\setminus\{0\}\subseteq \operatorname*{int}T$. $\Rightarrow
s_i^j>0\quad\forall i\in \{ 1,\ldots ,\ell\},\,j\in \{ 1,\ldots ,m\}$.\\
By \cite[Lemma 3]{weid93}, $F\cap (y^0-T)$ is bounded.
$\Rightarrow\exists w\in\mathbb{R}^\ell  \,:\;  F\cap (y^0-T)\subseteq
w+\operatorname*{int}\mathbb{R}_+^\ell$.\\
Consider first an arbitrary $n\in \{1,\ldots,\ell\}$.\\
$z_j^n:=\left\{
\begin{array}{c@{\mbox{ for }}l}
w_n & j=n,\\
y_j^0 & j\in \{ 1,\ldots,\ell\} \setminus \{ n \}.
\end{array}
\right.\quad $
$\Rightarrow y^0-z^n\in \mathbb{R}_+^\ell\setminus \{ 0 \} \subseteq H$.\\
Define $u^n\in \mathbb{R}^\ell$ by
$u_j^n:=\left\{
\begin{array}{c@{\mbox{ for }}l}
0 & j=n,\\
-1 & j\in \{ 1,\ldots,\ell\} \setminus \{ n \}.
\end{array}
\right.$\\
Since $H$ is open, there exists some $t_n>0$ with $(y^0-z^n)+t_nu^n\in
H$.\\
$w^n:=z^n-t_nu^n\in y^0-H$, and
$w_j^n=\left\{
\begin{array}{c@{\mbox{ for }}l}
w_n & j=n,\\
y_j^0+t_n & j\in \{ 1,\ldots,\ell\} \setminus \{ n \}.
\end{array}
\right.$\\
We now define $K:=\max \left( \max\limits_{i\in \{ 1,\ldots,\ell\} }\,
\frac{2(y_i^0-w_i)}{t_i} ,\, \max\limits_{i\in \{ 1,\ldots,\ell\} \atop j\in \{ 1,\ldots ,m\} }
(\frac{1}{s_i^j}\sum\limits_{r=1 \atop r\neq i}^\ell s_r^j +1)\right)
>0$.\\
For each $y\in D^K$, there exists some $i\in\{1,\ldots,\ell\}$ with
$y_i>0$ and
$y_i+Ky_j>0$ for all $j\in\{1,\ldots,\ell\}\setminus\{i\}$.
$\Rightarrow \sum\limits_{r=1}^\ell s_r^j y_r=s_i^j y_i+
\sum\limits_{r=1 \atop r\neq i}^\ell s_r^j y_r>y_i(s_i^j-\frac{1}{K}
\sum\limits_{r=1 \atop r\neq i}^\ell s_r^j )\ge 0 \quad\forall
j\in \{ 1,\ldots ,m\}$
because of the definition of $K$. Thus $D^K\subseteq T$.\\
Suppose: $\exists y\in F\cap (y^0-D^K)$.\\
$\Rightarrow y^0-y\in D^K$.\\
$\Rightarrow \exists i\in \{ 1,\ldots,\ell\}: \: y_i^0-y_i>0$ and
$y_i^0-y_i+K(y_j^0-y_j)>0$ for all $j\in \{ 1,\ldots,\ell\}
\setminus \{ i \}$.\\
Because of $y^0\in w+\operatorname*{int}\mathbb{R}_+^\ell$ and
$y\in F\cap(y^0-D^K)\subseteq F\cap(y^0-T)\subseteq w+\operatorname*{int}\mathbb{R}_+^\ell$, we have $w_i<y_i^0$ and $w_i<y_i$. This implies $2w_i<y_i+y_i^0$. $\Rightarrow -y_i<y_i^0-2w_i$. $\Rightarrow y_i^0-y_i<2(y_i^0-w_i)$.\\
Hence, $t:=\frac{y_i^0-y_i}{2(y_i^0-w_i)}\in (0,1)$.\\
$\Rightarrow y_i=y_i^0+2t(w_i-y_i^0)<y_i^0+t(w_i-y_i^0)$, since
$w_i-y_i^0<0$.\\
For each $j\in \{ 1,\ldots,\ell\} \setminus \{ i \}$, we get by $K\geq\frac{2(y_i^0-w_i)}{t_i}$:\\
$y_j^0-y_j>-\frac{1}{K}(y_i^0-y_i)\geq -t_i\frac{y_i^0-y_i}{2(y_i^0-w_i)}=-t_it$, hence $y_j<y_j^0+tt_i$.\\
$\Rightarrow y\in y^0+t(w^i-y^0)-\operatorname*{int}\mathbb{R}_+^\ell$.
$H$ convex, $0\in \operatorname*{bd}H$, $y^0-w^i\in H$. $\Rightarrow
t(y^0-w^i)\in H$.\\
$\Rightarrow y\in y^0-H-\operatorname*{int}\mathbb{R}_+^\ell =y^0-(H+\operatorname*{int}\mathbb{R}_+^\ell )\subseteq
y^0-H$.\\
$\Rightarrow y\in F\cap (y^0-H)$, a contradiction to $y^0\in
\operatorname*{Min}(F,H)$.
$\Rightarrow $ The supposition is wrong.\\
$\Rightarrow y^0\in \operatorname*{Min}(F,D^K)\subseteq \operatorname*{GMin}(F)$.
\end{itemize}
\end{proof}

\begin{remark}
The statement
\begin{eqnarray*}
\operatorname*{GMin}(F) & \supseteq & \{y^0\in F \mid  \exists \varphi : \mathbb{R}^{\ell } \to \mathbb{R}
\mbox{
strictly }
\mathbb{R} _+^\ell\mbox{--monotone, sublinear,}\\
& & \;\mbox{continuous}:\,\forall y\in F:\; \varphi (y-y^0)\geq 0\}
\end{eqnarray*}
also follows from Proposition 5.5 in \cite{dasa87}, where a corresponding assertion for properly efficient elements according to Benson \cite{bens79} had been proved in
partially ordered topological vector spaces.
\end{remark}

The inclusion in Theorem \ref{hab-t521}(b) cannot be replaced by an equation, which is illustrated by the following example.

\begin{example}\label{hab-e522}
Define $g:\mathbb{R}\to\mathbb{R}$ by
\[ g(x):=\left\{\begin{array}{l@{\quad\mbox{ for }\quad}l}
e^{x} -1 & x<0,\\x^2 + 2x& x\geq 0,\end{array}\right.\]
and $\varphi :\mathbb{R}^2\to\mathbb{R}$ by $\varphi ((y_1,y_2)^T)=g(y_1) + g(y_2)$.\\
Then $\varphi$ is continuous, strictly convex and strictly $\mathbb{R} _+^2$--monotone.\\
$F:=\{ (y_1,y_2)^T\in \mathbb{R}^2 \mid  y_1=1,\, y_2\leq -1\}\cup\{ (y_1,y_2)^T\in \mathbb{R}^2 \mid  y_1\in
[-1,1],\, y_1+y_2=0\}$.\\
$\varphi ((0,0)^T)=0<\varphi (y)\quad\forall y\in F\setminus \{ (0,0)^T \}$, but $\operatorname*{GMin}(F)=\emptyset$. 
\end{example}

The assumption in Theorem \ref{hab-t521}(c) is also not superfluous for the second equation.

\begin{example}\label{ex-NInotHe-folg}
Consider, like in Example \ref{ex-NInotHe}, $Y=\mathbb{R}^2$, the set $F:=\{ (y_1,y_2)^T \in \mathbb{R}^2\mid y_1<0,\, y_2=\frac{1}{y_1}\}+\mathbb{R}^2_+$ and $H:=- (Y\setminus \operatorname*{int}F)-(1,1)^T$. For $k:=(1,1)^T$ and $\varphi :=\varphi _{-H,k}$, we get 
$\operatorname*{Min}(F)=\operatorname*{argmin}_{F}\varphi$. This set coincides with the boundary of $F$ and contains more than one element. $\varphi $ is convex by \cite[Proposition 2.1]{Wei17a}, finite-valued and continuous by \cite[Theorem 3.1]{Wei17a}, strictly $\mathbb{R}_+^{\ell }$--monotone by \cite[Theorem 2.16]{Wei17a}. 
This implies, together with Corollary \ref{cor-jota1},
\begin{eqnarray*}
\operatorname*{Min}(F) & = & \{y^0\in F \mid  \exists \varphi :\mathbb{R}^{\ell}\to \mathbb{R}
\mbox{ strictly }
\mathbb{R}_+^\ell \mbox{--monotone, convex, continuous:}\\
 & & \varphi (y^0)=\min\limits_{y\in F}\, \varphi (y) \},
\end{eqnarray*}
though $\operatorname*{GMin}(F)=\emptyset$. 
\end{example} 

We get from \cite[Proposition 6]{Wei16b}:
\begin{eqnarray*}
\operatorname*{Min}(F) & \supseteq & \{y^0\in F \mid  \exists \varphi : \mathbb{R}^{\ell } \to \mathbb{R}
\mbox{ strictly } \mathbb{R} _+^\ell\mbox{--monotone, continuous}:\\
& & \;\forall y\in F\setminus\{y^0\}:\; \varphi (y) > \varphi (y^0)\}.
\end{eqnarray*}

Luc \cite[p.85]{Luc89b} illustrated by an example that the above inclusion cannot be replaced by an equation.
As we will show in Example \ref{hab-e523}, the elements $y^0$ in this inclusion are not necessarily unique minimizers of a convex strictly $\mathbb{R}_+^\ell $--monotone continuous functional on $F$. 

\begin{example}\label{hab-e523}
Consider $Y=\mathbb{R}^2$, 
$F=\{(y_1,y_2)^T\in \mathbb{R} ^2 \mid  -1\leq y_1\leq 0,\, -y_1\leq y_2\leq 1\}\cup\{ (y_1,y_2)^T\in
\mathbb{R}^2 \mid 
0\leq y_1\leq 1,\, -\frac{1}{2}\sqrt{y_1}\leq y_2\leq 1\}$,
$k=(1,1)^T$.\\
$H:=-(\{(y_1,y_2)^T\in \mathbb{R}^2 \mid  y_1\leq 0,\, y_1+2y_2<0\}\cup\{
(y_1,y_2)^T\in \mathbb{R}^2 \mid 
y_1>0,\, y_2<-\sqrt{y_1}\})$ is open, $(0,0)^T\in \operatorname*{bd}H$, $\mathbb{R}_+^2\setminus \{ (0,0)^T \} \subseteq
H$.
$\Rightarrow \varphi _{-\operatorname*{cl}H,k}$ is continuous and strictly $\mathbb{R}_+^2$--monotone by 
\cite[Theorem 3.1]{Wei17a} and \cite[Theorem 2.16]{Wei17a}.\\
$\varphi _{-\operatorname*{cl}H,k}((0,0)^T)=0<\varphi _{-\operatorname*{cl}H,k}(y)\quad\forall y\in F\setminus \{ (0,0)^T \}$.\\
If there would exist some strictly $\mathbb{R}_+^2$--monotone continuous convex
functional $\varphi :\mathbb{R}^2\to
\mathbb{R}$ with $\varphi (y)>\varphi ((0,0)^T)$
for all $y\in F\setminus\{(0,0)^T\}$, this would be equivalent to $(0,0)^T\in
\operatorname*{GMin}(F)$ by Theorem \ref{hab-t521}(c),
but this is not fulfilled.
\end{example}

Theorem \ref{hab-t521} implies together with Corollary \ref{cor-jota1}:

\begin{theorem}\label{t-Geoff-NI}
If {\rm (Sp-ex)} holds, then
$\operatorname*{GMin}(F)  = \operatorname*{NI-PMin}(F,\mathbb{R}_+^{\ell })$.
\end{theorem}

Thus, we get by Proposition \ref{hab-s527}:

\begin{theorem}\label{t-Geoff-NIcl}
If $F$ is a closed set, then $\operatorname*{GMin}(F)=\emptyset$ or
$\operatorname*{GMin}(F)  = \operatorname*{NI-PMin}(F,\mathbb{R}_+^{\ell })$.
\end{theorem}

Corollary \ref{c-NI-H} implies by Theorem \ref{t-Geoff-NI}:

\begin{corollary}
Assume {\rm (Sp-ex)}, $k\in \operatorname*{int}\mathbb{R}_+^{\ell }$.\\
Then $y^0\in \operatorname*{GMin}(F)$ if and only if there exists some closed convex set $H$ with $0\in \operatorname*{bd}H$ and $H+(\mathbb{R}_+^{\ell }\setminus \{ 0 \})\subseteq \operatorname*{int}H$ such that (a) or (b) holds. Note that (a) and (b) are equivalent to each other.
\begin{itemize}
\item[\rm (a)] $\varphi_{y^0-H,k}(y^0)=\operatorname*{min}_{y\in F}\varphi_{y^0-H,k}(y)=0$.
\item[\rm (b)] $\varphi_{-H,k}(y^0-y^0)=\operatorname*{min}_{y\in F}\varphi_{-H,k}(y-y^0)=0$.
\end{itemize}
Both functionals are finite-valued, continuous, convex and strictly $\mathbb{R}_+^{\ell }$--monotone.
\end{corollary}

\begin{proposition}\label{hab-s62_12}
Assume {\rm (Sp-ex)}, $a\in \mathbb{R}^\ell$, $k\in \mathbb{R}^\ell\setminus\{0\}$, and that $H$ is a proper, closed, convex subset of $\mathbb{R}^\ell$ with $\mathbb{R}^\ell=H+\mathbb{R}k$, $H+\mathbb{R}_>k\subseteq\operatorname*{int}H$ and
$H+(\mathbb{R}^\ell_+\setminus\{0\})\subseteq\operatorname*{int}H$.
Then $\operatorname*{argmin}_{F}\varphi _{a-H,k}\subseteq \operatorname*{GMin}(F)$.
\end{proposition}

\begin{proof}
By \cite[Prop. 4.5]{Wei17a} and \cite[Theorem 2.16]{Wei17a}, $\varphi _{a-H,k}$ is finite-valued, continuous, convex and strictly $\mathbb{R}^\ell_+$--monotone. 
The assertion follows from Theorem \ref{hab-t521}.
\end{proof}

Note that the assumption $\mathbb{R}^\ell=H+\mathbb{R}k$ in Proposition \ref{hab-s62_12} is satisfied, if the other assumptions and $k\in\operatorname*{int}0^+H$ hold.

\begin{proposition}\label{hab-s62_13}
\begin{itemize}
\item[]
\item[\rm (a)] Assume that $H\subset \mathbb{R}^\ell$ is a non-trivial, closed, convex cone  with
$\mathbb{R}^\ell _+\setminus\{0\} \subset \operatorname*{int}H$, $k\in \operatorname*{int}H$ and
$a\in \mathbb{R}^\ell$.
Then $\operatorname*{argmin}_{F}\varphi _{a-H,k}\subseteq \operatorname*{GMin}(F)$.
\item[\rm (b)] For each $y^0\in \operatorname*{GMin}(F)$, there exists some
non-trivial, polyhedral cone $H\subset \mathbb{R}^\ell$ with $\mathbb{R}^\ell _+\setminus\{0\} \subset \operatorname*{int}H$ such that $y^0$ is a unique minimizer of $\varphi _{y^0-H,k}$ on $F$ for each $k\in H\setminus\{0\}$.
\end{itemize}
\end{proposition}

\begin{proof}
\begin{itemize}
\item[]
\item[\rm (a)] 
$\operatorname*{argmin}_{F}\varphi _{a-H,k}\subseteq \operatorname*{WMin}(F,H)\subseteq \operatorname*{GMin}(F)$ by \cite[Theorem 5]{Wei16b} and by Proposition \ref{p-Geoff-cones}.
\item[\rm (b)] $y^0\in \operatorname*{GMin}(F)$ implies, by Proposition \ref{hab-s522}, $y^0\in \operatorname*{Min}(F,C^p)$ for some $p\in\mathbb{R}_>$. Apply \cite[Theorem 6]{Wei16b} to this efficient point set.
\end{itemize}
\end{proof}

The results can be applied to scalar optimization problems that generate properly efficient points according to Geoffrion. This was done for different approaches in \cite{weid93} and in \cite{Wei94}.

\def\cfac#1{\ifmmode\setbox7\hbox{$\accent"5E#1$}\else
  \setbox7\hbox{\accent"5E#1}\penalty 10000\relax\fi\raise 1\ht7
  \hbox{\lower1.15ex\hbox to 1\wd7{\hss\accent"13\hss}}\penalty 10000
  \hskip-1\wd7\penalty 10000\box7}
  \def\cfac#1{\ifmmode\setbox7\hbox{$\accent"5E#1$}\else
  \setbox7\hbox{\accent"5E#1}\penalty 10000\relax\fi\raise 1\ht7
  \hbox{\lower1.15ex\hbox to 1\wd7{\hss\accent"13\hss}}\penalty 10000
  \hskip-1\wd7\penalty 10000\box7}
  \def\cfac#1{\ifmmode\setbox7\hbox{$\accent"5E#1$}\else
  \setbox7\hbox{\accent"5E#1}\penalty 10000\relax\fi\raise 1\ht7
  \hbox{\lower1.15ex\hbox to 1\wd7{\hss\accent"13\hss}}\penalty 10000
  \hskip-1\wd7\penalty 10000\box7}
  \def\cfac#1{\ifmmode\setbox7\hbox{$\accent"5E#1$}\else
  \setbox7\hbox{\accent"5E#1}\penalty 10000\relax\fi\raise 1\ht7
  \hbox{\lower1.15ex\hbox to 1\wd7{\hss\accent"13\hss}}\penalty 10000
  \hskip-1\wd7\penalty 10000\box7}
  \def\cfac#1{\ifmmode\setbox7\hbox{$\accent"5E#1$}\else
  \setbox7\hbox{\accent"5E#1}\penalty 10000\relax\fi\raise 1\ht7
  \hbox{\lower1.15ex\hbox to 1\wd7{\hss\accent"13\hss}}\penalty 10000
  \hskip-1\wd7\penalty 10000\box7}
  \def\cfac#1{\ifmmode\setbox7\hbox{$\accent"5E#1$}\else
  \setbox7\hbox{\accent"5E#1}\penalty 10000\relax\fi\raise 1\ht7
  \hbox{\lower1.15ex\hbox to 1\wd7{\hss\accent"13\hss}}\penalty 10000
  \hskip-1\wd7\penalty 10000\box7}
  \def\cfac#1{\ifmmode\setbox7\hbox{$\accent"5E#1$}\else
  \setbox7\hbox{\accent"5E#1}\penalty 10000\relax\fi\raise 1\ht7
  \hbox{\lower1.15ex\hbox to 1\wd7{\hss\accent"13\hss}}\penalty 10000
  \hskip-1\wd7\penalty 10000\box7}
  \def\cfac#1{\ifmmode\setbox7\hbox{$\accent"5E#1$}\else
  \setbox7\hbox{\accent"5E#1}\penalty 10000\relax\fi\raise 1\ht7
  \hbox{\lower1.15ex\hbox to 1\wd7{\hss\accent"13\hss}}\penalty 10000
  \hskip-1\wd7\penalty 10000\box7}
  \def\cfac#1{\ifmmode\setbox7\hbox{$\accent"5E#1$}\else
  \setbox7\hbox{\accent"5E#1}\penalty 10000\relax\fi\raise 1\ht7
  \hbox{\lower1.15ex\hbox to 1\wd7{\hss\accent"13\hss}}\penalty 10000
  \hskip-1\wd7\penalty 10000\box7} \def\Dbar{\leavevmode\lower.6ex\hbox to
  0pt{\hskip-.23ex \accent"16\hss}D}
  \def\cfac#1{\ifmmode\setbox7\hbox{$\accent"5E#1$}\else
  \setbox7\hbox{\accent"5E#1}\penalty 10000\relax\fi\raise 1\ht7
  \hbox{\lower1.15ex\hbox to 1\wd7{\hss\accent"13\hss}}\penalty 10000
  \hskip-1\wd7\penalty 10000\box7} \def\cprime{$'$}
  \def\Dbar{\leavevmode\lower.6ex\hbox to 0pt{\hskip-.23ex \accent"16\hss}D}
  \def\cfac#1{\ifmmode\setbox7\hbox{$\accent"5E#1$}\else
  \setbox7\hbox{\accent"5E#1}\penalty 10000\relax\fi\raise 1\ht7
  \hbox{\lower1.15ex\hbox to 1\wd7{\hss\accent"13\hss}}\penalty 10000
  \hskip-1\wd7\penalty 10000\box7} \def\cprime{$'$}
  \def\Dbar{\leavevmode\lower.6ex\hbox to 0pt{\hskip-.23ex \accent"16\hss}D}
  \def\cfac#1{\ifmmode\setbox7\hbox{$\accent"5E#1$}\else
  \setbox7\hbox{\accent"5E#1}\penalty 10000\relax\fi\raise 1\ht7
  \hbox{\lower1.15ex\hbox to 1\wd7{\hss\accent"13\hss}}\penalty 10000
  \hskip-1\wd7\penalty 10000\box7}
  \def\udot#1{\ifmmode\oalign{$#1$\crcr\hidewidth.\hidewidth
  }\else\oalign{#1\crcr\hidewidth.\hidewidth}\fi}
  \def\cfac#1{\ifmmode\setbox7\hbox{$\accent"5E#1$}\else
  \setbox7\hbox{\accent"5E#1}\penalty 10000\relax\fi\raise 1\ht7
  \hbox{\lower1.15ex\hbox to 1\wd7{\hss\accent"13\hss}}\penalty 10000
  \hskip-1\wd7\penalty 10000\box7} \def\Dbar{\leavevmode\lower.6ex\hbox to
  0pt{\hskip-.23ex \accent"16\hss}D}
  \def\cfac#1{\ifmmode\setbox7\hbox{$\accent"5E#1$}\else
  \setbox7\hbox{\accent"5E#1}\penalty 10000\relax\fi\raise 1\ht7
  \hbox{\lower1.15ex\hbox to 1\wd7{\hss\accent"13\hss}}\penalty 10000
  \hskip-1\wd7\penalty 10000\box7} \def\Dbar{\leavevmode\lower.6ex\hbox to
  0pt{\hskip-.23ex \accent"16\hss}D}
  \def\cfac#1{\ifmmode\setbox7\hbox{$\accent"5E#1$}\else
  \setbox7\hbox{\accent"5E#1}\penalty 10000\relax\fi\raise 1\ht7
  \hbox{\lower1.15ex\hbox to 1\wd7{\hss\accent"13\hss}}\penalty 10000
  \hskip-1\wd7\penalty 10000\box7} \def\Dbar{\leavevmode\lower.6ex\hbox to
  0pt{\hskip-.23ex \accent"16\hss}D}
  \def\cfac#1{\ifmmode\setbox7\hbox{$\accent"5E#1$}\else
  \setbox7\hbox{\accent"5E#1}\penalty 10000\relax\fi\raise 1\ht7
  \hbox{\lower1.15ex\hbox to 1\wd7{\hss\accent"13\hss}}\penalty 10000
  \hskip-1\wd7\penalty 10000\box7}

\end{document}